\documentclass[11pt]{amsart}
\usepackage{amsmath}
\usepackage{amsfonts}
\usepackage{amsthm}
\usepackage{amssymb}
\usepackage{amscd}
\usepackage[all]{xy}
\usepackage{enumerate}
\usepackage{color}
\usepackage{bm}
\usepackage{amstext}
\usepackage{tikz-cd} 

\keywords{hyperelliptic curve, theta characteristic, moduli of curves, Fano threefold} 

\subjclass{primary 14H10; 14E30; secondary 14J45; 14N05}

\theoremstyle{plain}
\newtheorem{thm}{Theorem}[subsection]

\newtheorem{prop}[thm]{Proposition}

\newtheorem{cor}[thm]{Corollary}
\newtheorem{lem}[thm]{Lemma}

\theoremstyle{definition}
\newtheorem{defn}[thm]{Definition}

\newtheorem{gen}[thm]{Generality Conditions}

\newtheorem{nota}[thm]{Notation}

\newtheorem*{ackn}{Acknowledgements}

\newtheorem{rem}[thm]{Remark}
\newcommand{\sB}{\mathcal{B}}

\newcommand{\sH}{\mathcal{H}}

\newcommand{\sK}{\mathcal{K}}
\newcommand{\sL}{\mathcal{L}}

\newcommand{\sM}{\mathcal{M}}
\newcommand{\sO}{\mathcal{O}}
\newcommand{\sP}{\mathcal{P}}

\newcommand{\sR}{\mathcal{R}}
\newcommand{\sS}{\mathcal{S}}

\newcommand{\mC}{\mathbb{C}}

\newcommand{\mG}{\mathbb{G}}

\newcommand{\mP}{\mathbb{P}}

\numberwithin{equation}{section}

\newenvironment{fcaption}{\begin{list}{}{
\setlength{\leftmargin}{35pt}
\setlength{\rightmargin}{35pt}
\setlength{\labelsep}{5pt}
}}{\end{list}}

\author{Francesco Zucconi}

\address{D.M.I.F. \\
the University of Udine\\
Udine, 33100, Italy\newline
\texttt{francesco.zucconi@uniud.it}}

\begin{document}

\begin{center}
\textbf{\Large
The rationality of the moduli space\\ of two-pointed ineffective spin hyperelliptic curves}
\par
\end{center}
{\Large \par}

$\;$

\begin{center}
Francesco Zucconi 
\par\end{center}

\vspace{5pt}

$\;$

\begin{fcaption} {\small  \item 
Abstract. 
By the geometry of the $3$-fold quadric we show that the coarse moduli space of 
genus $g$ ineffective spin hyperelliptic curves with two marked points is a rational variety for every $g\geq 2$.
}\end{fcaption}

\vspace{0.5cm}


\markboth{Zucconi}{Moduli of two-pointed ineffective spin hyperelliptic curves}


\section{Introduction}
\subsection{The result}
We work over $\mC$, the complex number field. The purpose of this paper is to show the following result:
\medskip

\noindent
{\bf{Main Theorem}}  {\it{The coarse moduli space $\sS^{+,{\rm{hyp}}}_{g,2}$ 
of genus $g$ hyperelliptic ineffective spin curves with two marked points is an irreducible rational variety.}}
\medskip

\subsection{Motivations} 
 It is well known that the coarse moduli space of hyperelliptic curves $\sH_g$ is rational: \cite{Kat}, c.f. \cite{Bo}. In \cite{Ca}
 it is shown that  the coarse moduli space of $n$-marked hyperelliptic curves is irreducible for every $n$ and it is rational for every $n\leq 2g+8$. In \cite{Be} it is shown that $\sH_{g,n}$ is uniruled for $n \leq 4g+4$.

 In \cite{TZ4} we have shown that the coarse moduli space $\sS^{+,{\rm{hyp}}}_{g,1}$ 
 of ineffective spin hyperelliptic curves with one marked point is rational. Hence it is a natural question to study the rationality of the scheme $\sS^{+,{\rm{hyp}}}_{g,n}$ for $0\leq n\leq 2g+8$.

Our result fits into a vast literature concerning the rationality problem
 of special subloci of moduli spaces of curves too; see the book \cite{Boe}. For example in \cite[Proposition 2.2.1.5]{Boe},  it is studied the function field of the moduli space of ineffective spin curves $[C,\theta]$ where $C$ is a plane curve of degree $d$. We also like to recall that  ineffective spin hyperelliptic curves are crucial to construct hyperelliptic $K_3$ surfaces and then families of Godeaux surfaces as in \cite{Co}. They also play a role in Mumford's solution of the Hyperelliptic Schottky Problem \cite{Mu2}.  Families of hyperelliptic curves with two marked points play a role also in arithmetic; for example see: \cite{Sh}. 

There is also another reason to shed some light on the geometry of some loci of the moduli space of spin curves. Indeed it seems to exist a structural relation between spin curves and $3$-folds geometry. In \cite[Cor. 4.1.1]{TZ1} we showed that the geometry of trigonal spin curves is ruled by the geometry of rational curves on the del Pezzo threefold $B$ of degree $5$ (and index $2$). In \cite[Prop.~3.1.2]{TZ2} we constructed a theta characteristic on the general trigonal curve from
the incidence correspondence of intersecting lines on $B$. Indeed we generalised and we extended Mukai's approach to study genus twelve prime Fano threefold $V_{22}$: \cite{Mu2, Mukai12}. 
Finally we stress that the result of this paper is directly related to the rationality results of \cite{TZ4} and of \cite{TZ3}.

\subsection{Spin curves}
A smooth spin curve is a pair $(C,\theta)$ of a smooth curve of genus $g\geq 2$ and a line bundle $\theta$ on $C$ such that $\theta^{\otimes 2}$ is isomorphic to the canonical bundle $\omega_C$. The coarse moduli space $\sS_g$ of such pairs has a compactification ${\overline{S_g}}$, see: \cite{cornalba}, which is compatible with the Deligne-Mumford compactification ${\overline{M_g}}$ of the coarse moduli space $M_g$ of smooth curves of genus $g$ via stable curves \cite{DM}. For its geometry see: \cite{Fa}.

The natural forgetful  morphism $\pi\colon \sS_g\to M_g$ is a finite map of degree $2^{2g}$. By \cite{Mu1}, \cite{At} we know that $\sS_g$ is the disjoint union of two irreducible components $\sS^{+}_{g}$ and $\sS^{-}_{g}$ where $\sS^{+}_{g}$ is the moduli space of those $[C,\theta]$ such that $h^{0}(C,\theta)$ is an even number and  $\sS^{-}_{g}$ is the one where $h^{0}(C,\theta)$ is an odd number. 

Those $[C,\theta]\in \sS^{+}_{g}$
with $h^{0}(C,\theta)=0$ fill an open subset inside $\sS^{+}_{g}$ and the class $\theta$ is said to be an ineffective theta characteristic on $C$. 

The geometry of  $\sS^{}_{g}$ is a well-established subject of study since the beginnings of algebraic geometry. The hyperelliptic case has been considered in the literature; see for example: \cite{Na}, \cite{Ol}.

\subsection{  $\sS^{+,{\rm{hyp}}}_{g,2}$ and the quadric $3$-fold}
In this paper we consider the moduli space $\sS^{+,{\rm{hyp}}}_{g,2}\hookrightarrow \sS^{+}_{g,2}$ given by the classes $[C,\theta, m,n]$ where $C$ is a smooth hyperelliptic curve of genus $g$, $h^0(C,\theta)=0$ and $m,n\in C$ up to automorphism. We show below that the geometry of the $3$-fold quadric $Q\subset\mP^4$ encodes the one of $\sS^{+,{\rm{hyp}}}_{g,2}$. 
\subsubsection{Linear algebra set up}
Let $V$ be a $5$-dimensional vector space and let $Q\subset \mP^4=\mP(V)$ be a smooth quadric threefold. Consider the couple $(Q,q)$ where $q\subset Q$ is a smooth conic. Let $\mP(W)<\mP^4$ be the projective plane spanned by $q$ where $W <V$ is the corresponding $3$-dimensional vector sub-space. We consider $\Phi_Q\colon V\to V^\vee$ the natural isomorphism to the dual space $V^{\vee}$ of $V$ induced by anyone among the  non-degenerate bilinear form $b\colon V\times V\to\mathbb C$ associated to $Q$. Since $q$ is smooth and $W$ is not inside $Q$, it holds that :$$V=W\oplus^{\perp_b}W^{\perp}$$
where $W^{\perp}:=\{v\in V|\forall w\in W,\,\, b(v,w)=0 \}.$
\subsubsection{The hyperelliptic curve}

We take a general element $[H]\in \mP(V^{\vee})$. The precise constrains on $H$ are stated in Generality Conditions \ref{generalityconditionsonH}.

Let $Q_{H}\subset Q$ be the hyperplane section $Q\cap H$.
We denote by $|(1,0)|$, $|(0,1)$ the linear systems which induce, respectively, the two natural rulings.

 Inside $Q_H$ we take a general $R\in |(1,d-1)|$; the precise constrains on $R$ are stated in Generality Conditions \ref{generality conditions}. Definitley $R$ is a rational curve of  degree $d$ with respect to the embedding $Q_H\hookrightarrow H\subset \mP^4$ given by $|(1,1)|$. 
 
 Since $q$ is of degree $2$ and since the tangent hyperplane section $Q_t:=T_{t}Q\cap Q$ to $Q$ at any point $t\in Q$ is a cone over a  smooth conic, with vertex $t$, it holds that for any $t\in Q$ there exist two points $a(t), a'(t)\in q$, not necessarily distinct,  such that the lines $\langle t,a(t)\rangle, \langle t,a'(t)\rangle\subset Q$. In particular by a general point $t\in R$ there pass (at most) two lines $l(t)$, $l'(t)$ which satisfy the following conditions: 
\begin{enumerate}[{i)}]
\item $l(t), l'(t)\subset Q$;
\item $\emptyset\neq l(t)\cap q=\{ a(t) \}$, 
\item $\emptyset\neq l'(t)\cap q=\{a'(t)\}.$
\end{enumerate}

By the same reason, given a point $a\in q$ there exist (at most) $d$ lines $l_{a}^{1},..., l_{a}^{d}\subset Q$ passing through $a$ such that $l_{a}^{1}\cap R=\{ t^{1} \}$,..., $l_{a}^{d}\cap R=\{ t^{d} \}$.

 This picture shows a $[2,d]$ correspondence $C\subset R\times q$. The scheme $C$ comes equipped with a $g^1_2$ given by the morphism $C\to R$ induced by the natural projection $R\times q\to R$; the other projection induces a $g^1_d$ on $C$. 
 
 \subsubsection{The ineffective theta-characteristic}
 
 By the geometry of $Q$ we can also describe explicitely an ineffective theta-characteristic on $C$. To see it we introduce the notion of marked line for the rational curve $R$ with respect to the triple $(Q,q, H)$.

 {\it{The notion of $R$-marked line.}} A $R$-marked line is a point $[t,a]\in R\times q$ such that the line $l_{[t,a]}:=\langle t,a\rangle\subset\mP^4$ is actually a line inside $Q$. The line $l_{[t,a]}$ is called the support of the marked line $[t,a]$. We can show that the correspondence $C$ is the scheme of $R$-marked lines with respect to $(Q,q,H)$. 
 
 Now consider a general market line $l=l_{[t,a]}$, then by the projection from $l\subset Q\subset\mP^4$ we find that there exist exactly $d-1$ marked lines  $[t_1, a_1],..., [t_{d-1}, a_{d-1}]$ such that 
$l_{[t_i, a_i]}\cap l\neq\emptyset$ and such that  $l_{[t_i, a_i]}\cap l$ is distinct from both $l\cap q$ and $l\cap R$. In Proposition \ref{thetatheta} we prove that there exists an ineffective theta characteristic $\theta(R)$ on $C$ such that the unique effective divisor of  $|\theta(R)+[t,a]|$ is exactly  $\sum_{i=1}^{d-1}[t_i, a_i]$.

\subsubsection{The two special marked lines}
Our construction comes with some constrains which gives two special points on $C$. Let $\{x,y\}=q\cap Q_H$ and let us consider the two linear series $|(1,0)|$ and $(0,1)|$ on $Q_H$. 
Above we have selected a general element  $[R]\in |(1,d-1)|$. In particular there exists a unique line $m\in|(0,1)|$ such that $m\cap q=\{x\}$ and there exists a unique line $n\in|(0,1)|$ such that $n\cap q=\{y\}$. 
Obviously there exists a unique point $p_x(R)\in R\cap m$ and there exists a unique point $p_y(R)\in R\cap n$. For a general $R$ we denote by $m(R):=[p_x(R),x]$, $n(R):=[p_y(R),y]$ these two special $R$-marked lines. 
Note that if we vary $R$ the points  $p_x(R)$, $p_y(R)$ vary as well, but the respective supporting lines of $m(R)$, $n(R)$ do not vary. These special points on $C(R)$ originate the class $[C(R),\theta(R), m(R),n(R)] \in\sS^{+,{\rm{hyp}}}_{g,2}.$
\begin{figure}
\centering
\includegraphics[scale=0.5]{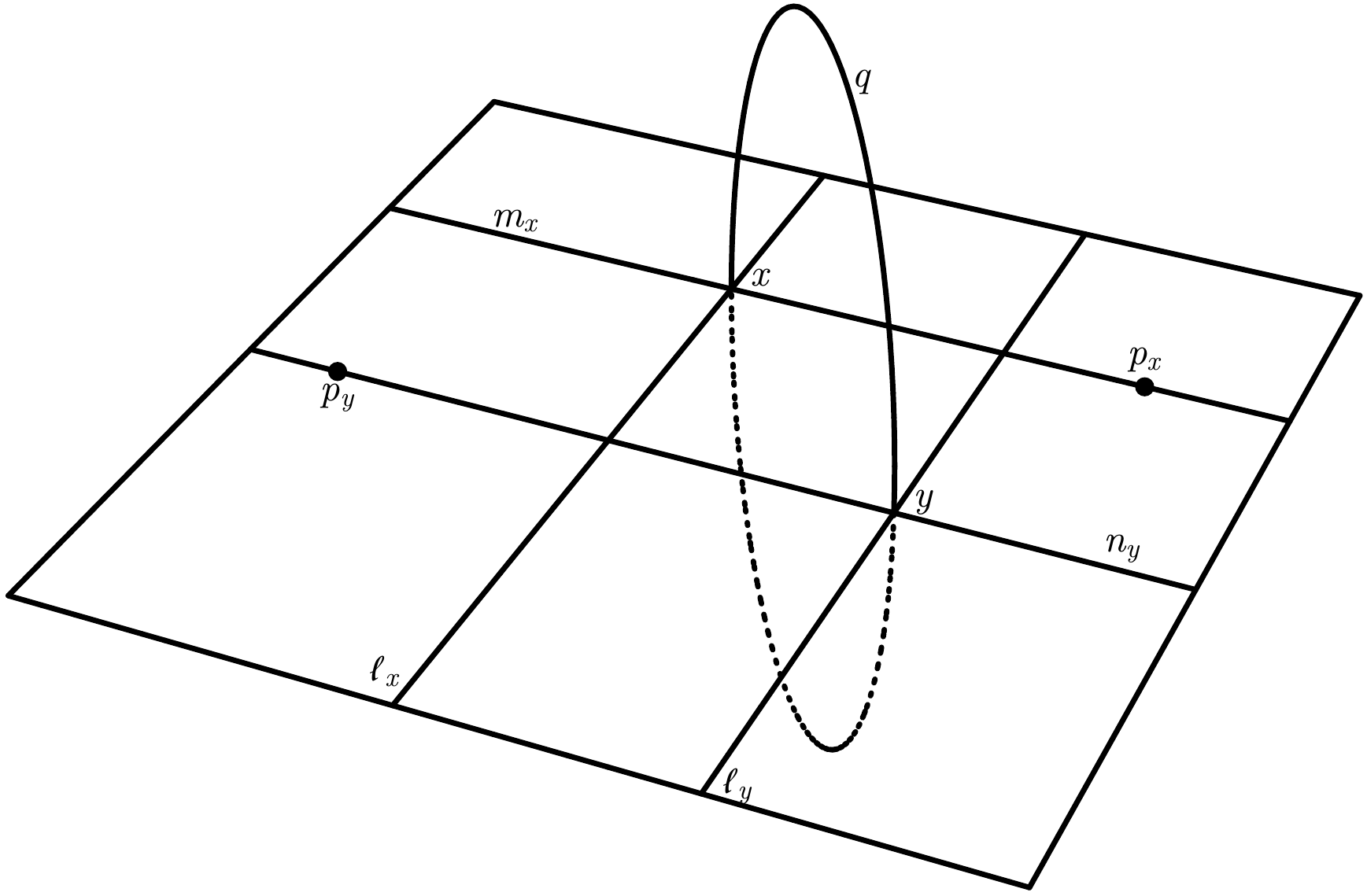}
\caption{The surface $Q_H$}
\label{etichettaPerRiferimenti}
\end{figure}

\subsection{ On the proof of the Main Theorem} 
The curve $C$ naturally comes equipped with a forgetful morphism $f_R\colon C\to {\rm{Hilb}}^{Q}_1$ obtained associating to any marked line the corresponding line of $Q$. The image $M$ is a singular curve with two points of multiplicity $d-1$; 
actually a line $l\subset Q$ can be the support of more than one marked line; see Remark \ref{multisupporto}. Moreover $M$ is contained in the quadric $S_q\subset{\rm{Hilb}}^{Q}_{1}$ which parameterises the lines of $Q$ which touch $q$. 

On the other hand a general element $[C,\theta, m,n]\in \sS^{+,{\rm{hyp}}}_{d-1,2}$ comes equipped with a surface $Q_{C,\theta,m,n}$ isomorphic to $\mP^1\times\mP^1$ and a morphism $f_{\theta,m,n}\colon C\to S_{C,\theta,m,n}\subset\mP^3$ where $S_{C,\theta,m,n}$ is the image of $Q_{C,\theta,m,n}$ given by its rulings; see Corollary \ref{Serveassai}.

In Proposition \ref{identificazione} we build an identification between $({\rm{Hilb}}^{Q}_1,  S_q)$ and $(\mP^3, S_{C(R),\theta(R),m(R),n(R) })$, where now we have stressed the dependence of the hyperelliptic curve by the rational curve $R$. Actually we can read off the full geometry of a general element $[C,\theta, m,n]\in \sS^{+,{\rm{hyp}}}_{d-1,2}$ via the above identification. This leads to the Reconstruction Theorem, which asserts that given $[C,\theta, m,n]$ there exists $R\subset Q_H$ as above such that  $[C(R),\theta(R), m(R),n(R)]=[C,\theta, m,n]$. 
Finally we need a detailed but simple analysis of the automorphism group of $(Q,q,H)$ and of its action on the linear system $|(1,d-1)|$ of $Q_H$; actually we show that we have to consider a $\frac{\mathbb Z}{2\mathbb Z}$-action on $|(1,d-1)|$. In Corollary \ref{loe} we show that 
$|(1,d-1)|// \frac{\mathbb Z}{2\mathbb Z}$ is rational. Finally we show the Injectivity Theorem, see Theorem \ref{injectivity}, which claims that  $|(1,d-1)|//\frac{\mathbb Z}{2\mathbb Z}$ is birational to  $\sS^{+,{\rm{hyp}}}_{d-1,2}$ and the Main Theorem follows; see Theorem \ref{theend}.

\begin{ackn}
The author thanks Hiromichi Takagi for very useful discussions and Gianluca Gorni since he has done the three pictures which should strongly help the reader to follow the arguments. This research is supported by DIMAGeometry PRIDZUCC.
\end{ackn}

\section{Classical and less classical results on hyperelliptic curves}

\subsection{Ineffective Spin Hyperelliptic curves with a marked point}
 Let $C$ be an hyperelliptic curve of genus $g$ and let $W(C)$ be the set of its Weierstrass points.
 \subsubsection{Partitions of $W(C)$} It is easy to see that if we select a $g+1$ partition 
 $$\sP_{\theta}:= \{ \{w_1,..w_{g+1}\},\{w'_{1},..,w'_{g+1}\}  \}$$ 
 \noindent of $W(C)$, that is $W(C)=\{w_1,..w_{g+1}\} \bigsqcup \{w'_{1},..,w'_{g+1}\}$, then $w_1+w_2+...+w_{g+1}$ is linearly equivalent to $w'_{1}+...+ w'_{g+1}$. Moreover the divisor 
 \begin{equation}
\label{eq:twopre}
\theta\sim w_1+\cdots+w_{g+1}-g^1_2\sim w'_{1}+\cdots+w'_{g+1}-g^1_2
\end{equation}
\noindent gives an ineffective theta characteristic $[\theta]$ inside ${\rm{Pic}}(C)$, where by $g^1_2$ we will denote both the linear system $|p+p'|$ giving the $2$-to-$1$ cover $\phi_{|p+p'|}\colon C\to\mP^1$ than a divisor of the linear system. It is well-known that also the viceversa is true.

\begin{prop}\label{formoftheta} Ineffective theta characteristics are in one-to-one correspondence to $g+1$ partions of $W(C)$.
\end{prop}
\begin{proof} See: cf.\cite[p. 288, Exercise 32]{ACGH}. 
\end{proof}

\subsubsection{The case with one marked point} We recall that $\sS^{+,{\rm{hyp}}}_{g,1}$ is the coarse moduli space of $1$-marked ineffective spin hyperelliptic curves. We will need the following:
\begin{lem}
\label{lem:reconst}
Let $[(C,\theta, m)]$ be any element of $\sS^{+,{\small{\rm{hyp}}}}_{g,1.}$
Let $\sP_{\theta}$ be a $g+1$-partition of $W(C)$ as above. The following assertions hold\,$:$
\begin{enumerate}[$(1)$]
\item The linear system $|\theta+g^{1}_{2}+m|$ defines a birational morphism $$\varphi_{|\theta+g^{1}_{2}+m|}:=\psi\colon C\to \mP^2$$
from $C$ to a plane curve $M_{\theta,m}$ of degree $g+2$. 
\item Let $D_m=n_1+n_2+...+n_g$ be the unique element of
$|\theta+m|$. Then there exists $o\in M_{\theta,m}$ such that $\psi (n_1)=\psi (n_2)=...=\psi( n_g)=o$.
\vspace{5pt}

\noindent For the assertions $(3)$ and $(4)$, we set $A:=\{m, w_1,\dots, w_{g+1}, w'_{1},\dots, w'_{g+1}\}$.
\item
The support of $D_m$ contains no point of $A$.
\item
The point $o$ as in $(2)$ is different from the $\psi$-images of points of $A$. 
Besides no two points of $A$ are mapped to the same point by the map $\psi\colon C\to \mP^2$. 
\item The curve $M_{\theta,m}$ has a point of multiplicity $g$ supported on $o$.
\item There exist two lines $L_{\theta,m}, L^{'}_{\theta,m}\subset \mP^2$ such that $o\not\in L_{\theta,m}\cup L^{'}_{\theta,m}$ and $$L_{\theta,m}\cap M(C)=\{\psi (w_1),...,\psi(w_d),\psi(m)\},$$
$$L^{'}_{\theta,m}\cap M(C)=\{\psi(w'_1),...,\psi(w'_d),\psi(m)\}.$$
\end{enumerate}
\end{lem}
\begin{proof}See: c.f. \cite[Lemma 4.2.2]{TZ4}.
\end{proof}

%

\begin{rem}\label{polyhedra} By Proposition \ref{formoftheta} an ineffective theta characteristic $[\theta]$ is invariant by the action of the hyperelliptic involution. The unique effective divisor $a_1(P)+a_2(P)+...+a_g(P)\in |\theta+P|$ is called {\it{the theta polyhedron associated to the point $P\in C$}}. It holds that
 $a'_1(P)+a'_2(P)+...+a'_g(P)\in |\theta+P'|$ is the theta polyhedron associated to $P'$ where $P+P'\sim a_i(P)+a'_i(P)$, $i=1,..., g$ is the hyperelliptic linear series.
\end{rem}

\subsection{Ineffective Spin Hyperelliptic curves with two marked point} 
We fix two general points $m,n\in C$. In the rest of the paper we set $g=d-1$. 

\subsubsection{Theta characteristics and couples of points} If $\theta$ is ineffective
then $h^0(C,\sO_C(\theta+m))=h^0(C,\sO_C(\theta+n))=1$. Moreover there exist $2d-2$ mutually distinct points $m_1,...,m_{d-1}$, $n_1$,..., $n_{d-1}\in C$ such that 
$$\sum_{i=1}^{d-1}n_i\in |\theta+m|$$
$$ \sum_{i=1}^{d-1}m_i\in |\theta+n|.$$
\begin{nota}We set 
\begin{enumerate}[{1)}]
\item $\sL_{\theta+m+n}:=\sO_C(m+n+\theta));$
\item $\sL_{W}:=\sO_C (\theta+g^1_2),$
\end{enumerate}
and we stress that $\sL_W$ is  $\sO_C(w_1+w_2+...+w_d).$
\end{nota}
\begin{lem}\label{fatt}It holds:
\begin{enumerate}
\item $\sL_{\theta+m+n}$ is base point free and $h^0(C,\sL_{ \theta+m+n})=2$;
\item $\sL_W$ is base point free and $h^0(C,\sL_{ W})=2$;
\item $n_1+...+n_{d-1}+m'\in|w_1+...+w_d|\ni m_1+...+m_{d-1}+n'$.
\end{enumerate}
\end{lem}
\begin{proof} $(1)$. The divisor $m+\sum_{i=1}^{d-1}m_i$ is linearly equivalent to $n+\sum_{i=1}^{d-1}n_i$. Moreover $K_C-\theta-m-n\sim \theta-m-n$. By Riemann- Roch theorem it follows that $h^0(C,\sL_{ \theta})=2$. Since $m'\neq n$ it is easy to see that the linear system $|\theta+n+m|$ is a $g^1_d$. $(2)$ is well-known and easy to be proved. To show $(3)$ consider the unique $g^1_2$ on $C$: $|m+m'|$ and note that $w_1+w_2+...+w_d\sim\theta+m+m'\sim n_1+...+n_{d-1}+m' \sim m_1+...+m_{d-1}+n'$. 

 \end{proof}

\noindent Since $2\theta\sim K_C$ we have that for $\eta\in |K_C+p+p'+m+n|$ it holds that
\begin{equation}\label{itiseta}\eta\sim\theta+p+p'+n+\sum_{i=1}^{d-1}n_i\end{equation}
\noindent
and $$\sO_C(\eta)= \sL_{\theta+m+n}\otimes \sL_{W}.$$

\subsubsection{The case with two marked points}  We recall that $\sS^{+,{\rm{hyp}}}_{g,2}$ is the coarse moduli space of $2$-marked ineffective spin hyperelliptic curves.
The reader can easily follow the proof of the following Lemma, by considering the case where $C$ has genus $2$, that is when $d=3$.

\begin{lem}\label{degreedlinear} Let $[C,\theta, m,n]\in \sS^{+,{\rm{hyp}}}_{d-1,2}$ be a general element. It holds that:
 \begin{enumerate}[$(1)$]
\item the map $\phi_{|\eta|}\colon C\to \mP^{d+1}$ is an embedding;
\item the linear span $\langle\phi_{|\eta|}(m_1),...\phi_{|\eta|}(m_{d-1}), \phi_{|\eta|}(n_1),...\phi_{|\eta|}(n_{d-1})\rangle$ is a $\mP^{d-1}$;
\item the linear spans $\Pi:=\langle\phi_{|\eta|}(m_1),...\phi_{|\eta|}(m_{d-1})\rangle$, $\Pi':=\langle\phi_{|\eta|}(n_1),...\phi_{|\eta|}(n_{d-1})\rangle$ are distinct and both of dimension $d-2$;
\item the linear space $\Theta:=\Pi\cap\Pi'$ is disjoint from $\phi_{|\eta|}(C)$;
\item the composition $f:=\pi_{\Theta}\circ \phi_{|\eta|}\colon C\dashrightarrow \mP^{3}$, where $\pi_{\Theta}\colon   \mP^{d+1}\setminus \Theta\to\mP^3$ is the projection from $\Theta$, is a morphism to a curve $M\subset\mP^3$ such that
$f(m_1)=f(m_2)=...=f(m_{d-1})=a_{m}\in M$, $f(n_1)=f(n_2)=...=f(n_{d-1})=a_{n}\in M\subset\mP^3$ and $a_m,a_n$ are respectively the support of a point of multiplicity $d-1$;
\item there exist two disjoint lines $L$, $L'$ inside $\mP^3$ such that  $L\cap f(C)=\{f(w_1),...,f(w_{d})\}$ and $L'\cap f(C)=\{f(w'_1),...,f(w'_{d})\}$;

\end{enumerate}
\end{lem}
\begin{proof} 
 By Riemann-Roch's theorem it holds that $\phi_{|\eta|}\colon C\to\mP^{d+1}$ is an embedding. 
It is easy to show that 
$$
\eta-\sum_{i=1}^{d-1}n_i-\sum_{i=1}^{d-1}m_i\simeq\theta +p+p'+n-\sum_{i=1}^{d-1}m_i\simeq p+p'.
$$
\noindent 
Hence $h^0(C,\sO_C(\eta-\sum_{i=1}^{d-1}n_i-\sum_{i=1}^{d-1}m_i))=2$. By Geometric Riemann Roch theorem this implies that the linear span of $\phi_{|\eta|}(m_1)$,..., $\phi_{|\eta|}(m_{d-1})$,$ \phi_{|\eta|}(n_1$),..., $\phi_{|\eta|}(n_{d-1})$ is a $\mP^{d-1}<\mP^{d+1}$. By Lemma \ref{lem:reconst} $(1)$ $h^0(C,\sO_C(\eta-\sum_{i=1}^{d-1}m_i))=3$. It follows that $\phi_{|\eta|}(m_1),..., \phi_{|\eta|}(m_{d-1})$ generates a $d-2$ linear subspace $\Pi$ of $\mP^{d-1}$ and by the analogue reason $ \phi_{|\eta|}(n_1),..., \phi_{|\eta|}(n_{d-1})$ generates a $d-2$ linear subspace $\Pi'<\mP^{d-1}$. By the Grassmann formula the subspace $W$ of $H^0(C,\sO_C(\eta))$ generated by the subspaces $H^0(C,\sO_C( \eta-\sum_{i=1}^{d-1} m_i))$ and $H^0(C,\sO_C( \eta-\sum_{i=1}^{d-1} n_i))$ has dimension $4$. Then $\rm{dim}{\rm{Ann}}(W)=d-2$ and $\rm{dim}H^0(C,\sO_C( \eta))^{\vee}/{\rm{Ann}}(W)=4$. Denote by $\Theta:=\mP({\rm{Ann}}(W))$ and set $U^{\perp}:= H^0(C,\sO_C( \eta))^{\vee}/{\rm{Ann}}(W)$. Then $\Theta=\Pi\cap \Pi'$ is of dimension $d-3$.

Let us consider the linear projection $\pi_{\Theta}\colon \mP^{d+1}\dashrightarrow \mP(U^{\perp})=\mP^3$ from the subspace $\Theta$. We set $f\colon C\dashrightarrow \mP^3$ to be the composition of $\phi_{|\eta|}$ followed by the projection $\pi_{\Theta}$. It is easy to show that $\phi_{|\eta|}(m_i)\not \in\Theta$ and that $\phi_{|\eta|}(n_i)\not \in\Theta$ and more generally that $\Theta\cap\phi_{|\eta|}(C)=\emptyset$ since Lemma \ref{fatt} $(1)$ and the remark that $| n+\sum_{i=1}^{d-1}n_i |=|\eta-\theta-g^1_2|$ is base point free.

 This implies that $f\colon C\to\mP^3$ is a morphism. Moreover since the span of $\phi_{|\eta|}(m_1),..., \phi_{|\eta|}(m_{d-1})$ is of dimension $d-2$ and it contains $\Theta$  then $f(m_1)=f(m_2)=...=f(m_{d-1})=a_{m}\in M\subset\mP^3$ and by the same argument $f(n_1)=f(n_2)=...=f(n_{d-1})=a_n\in M\subset\mP^3$. Moreover $a_m\neq a_n$ and it clearly holds that $M$ has degree $2d$. Then $M$ is a singular curve with two points of multiplicity $d-1$ and $f\colon C\to M$ is the normalisation morphism since $C$ is of genus $d-1$.

Since $\eta- (w_1+...+w_d)\sim \eta-( w'_1+...+w'_d)$ is the $g^1_d$ given by $\sL_{\theta}$ it follows that $\phi_{|\eta|}(w_1),..., \phi_{|\eta|}(w_{d})$ generate a $d-1$ linear subspace and analogously for $\phi_{|\eta|}(w'_1),..., \phi_{|\eta|}(w'_{d})$. Since  $\eta- (w_1+...+w_d +\sum_{i=1}^{d-1}m_i)$ is effective then the images of the Weierstrass points $f(w_1)$,..., $f(w_d)$ belong to a line $L\subset\mP^3$ and analogousely  $f(w'_1)$,..., $f(w'_d)$ belong to a line $L'\subset\mP^3$. More precisely, by the fact that
$$n+\sum_{i=1}^{d-1}n_i \in|\eta-w_1+...+w_d|\ni m+\sum_{i=1}^{d-1}m_i$$ 
it holds that $\phi_{|\eta|}(w_1),..., \phi_{|\eta|}(w_d)$ span a $d-1$-plane inside $\mP^{d+1}$ and that $\phi_{|\eta|}(w_1)$,...,$ \phi_{|\eta|}(w_d)$, $\phi_{|\eta|}(m)$ span a hyperplane section $\Pi_m$ which contains $\phi_{|\eta|}(m_1),..., \phi_{|\eta|}(m_{d-1})$ while $\phi_{|\eta|}(w_1),..., \phi_{|\eta|}(w_d),\phi_{|\eta|}(n)$ span a hyperplane section $\Pi_n$ which contains $\phi_{|\eta|}(n_1),..., \phi_{|\eta|}(n_{d-1})$. Hence the $d-1$-plane spanned by  $\phi_{|\eta|}(w_1)$,..., $\phi_{|\eta|}(w_d)$ contains $\Theta$. The same holds for $\phi_{|\eta|}(w'_1)$,..., $\phi_{|\eta|}(w'_d)$.

We claim that $L\cap L'=\emptyset$. By contradiction assume that $L\cap L'\neq \emptyset$. Then there exists a $\mP^{d-2}$ containing $\Theta$ such that $\mP^{d-2}\subset \langle \phi_{|\eta|}(w_1),..., \phi_{|\eta|}(w_d)\rangle$ and $\mP^{d-2}\subset \langle \phi_{|\eta|}(w'_1),..., \phi_{|\eta|}(w'_d)$. Then there exists a $w_i$ and a $w'_j$, $1\leq i,j\leq d$, such that $\langle \mP^{d-2}, w_i,w'_j\rangle$ is an hyperplane.  Since $\langle\Theta, \phi_{\eta}(w_i)\rangle=\langle \phi_{|\eta|}(w_1),..., \phi_{|\eta|}(w_d)\rangle $ and $\langle\Theta, \phi_{\eta}(w'_i)\rangle=\langle \phi_{|\eta|}(w'_1),..., \phi_{|\eta|}(w'_d)\rangle$ it follows that $\eta\sim \sum_{s=1}^{d}w_s+ \sum_{s=1}^{d}w'_s$. Then it easily follows that $\sum_{s=1}^{d}w'_s\sim n+\sum_{i=1}^{d-1}n_i$. Since $\sum_{s=1}^{d}w'_s\sim \theta+g^1_2$ and since $n+\sum_{i=1}^{d-1}n_i\sim n+m+\theta$ it follows that $n+m$ is the $g^1_2$: a contradiction.
\end{proof}

\subsection{The quadric associated to a general element of  $\sS^{+,{\rm{hyp}}}_{g,2}$}
We consider the rational map induced by $\sL_{\theta+m+n}$ and respectively by $\sL_W$: $$\phi_{\sL_{\theta+m+n}}\colon C\dashrightarrow \mP(H^0(C,\sL_{ \theta+m+n})^\vee),\, \phi_{\sL_W}\colon C\dashrightarrow\mP(H^0(C,\sL_{ W})^{\vee}).$$ By Lemma \ref{fatt} both are morphisms. We denote by
$$
\Phi\colon C\to\mP(H^0(C,\sL_{ \theta+m+n})^\vee)\times  \mP(H^0(C,\sL_{ W})^{\vee})
$$ the product morphism, that is $\Phi= \phi_{\sL_{\theta+m+n}}\times \phi_{\sL_W}$. We denote by $f_{\theta,m,n}\colon C\to \mP^3=\mP(U^{\perp})$ the morphism constructed in Lemma \ref{degreedlinear} $(5)$ too. We stress that the Lemma below is crucial for the rationality result. 

 \begin{lem} \label{quadica} Let $[(C,\theta,m,n)]\in  \sS^{+,{\rm{hyp}}}_{d-1,2}$ be a general element. Then the rational map $\Phi\colon C\to \mP(H^0(C,\sL_{ \theta+m+n})^\vee)\times \mP(H^0(C,\sL_{ W})^{\vee})$ is a morphism of degree $1$. Moreover there exists an embedding $\iota\colon \mP(H^0(C,\sL_{ \theta+m+n})^\vee)\times  \mP(H^0(C,\sL_{ W})^\vee)\to\mP^3=\mP(U^{\perp})$ such that $f_{\theta,m,n}=\iota\circ \Phi$.
  \end{lem}
  \begin{proof}
 By the equation (\ref{itiseta}) $\sL_{ W}\otimes \sL_{ \theta+m+n}$ is linearly equivalent to $\sO_C(\eta)$. 

We consider the standard multiplication map:
 $$
 \mu\colon H^0(C,\sL_{ \theta+m+n})\otimes H^0(C,\sL_{ W})\to H^0(C,\sO_C(\eta))
 $$
By the Castelnuovo's free pencil trick it is an injection. Set $$U:=\mu ( H^0(C,\sL_{ \theta+m+n}) \otimes H^0(C,\sL_{ W}))$$ and note that $U$ is isomorphic to $H^0(C,\sL_{ W})\otimes H^0(C,\sL_{ \theta+m+n})$.  By Lemma \ref{degreedlinear} $(3)$ the projectivization of the annihilator subspace of $U$ is $\Theta$. Hence $U^\vee=U^{\perp}$. This implies that for the projection from $\Theta$ it holds $\pi_{\Theta}\colon   \mP^{d+1}\setminus \Theta\to\mP^3=\mP(U^{\vee})$. By Lemma \ref{degreedlinear} the claim follows where 
$$\iota\colon  \mP( H^0(C,\sL_{ W})^\vee))\times\mP(H^0(C,\sL_{ \theta+m+n})^\vee))\to\mP(U^{\vee})$$ comes from the dual of the isomorphism $U\simeq H^0(C,\sL_{ W})\otimes H^0(C,\sL_{ \theta+m+n})$ induced by the multiplication map.
 \end{proof}

\begin{defn}\label{eccola}
The surface $\mP(H^0(C,\sL_{ \theta+m+n})^\vee) \times \mP (H^0(C,\sL_{ W})^\vee)$ is denoted by $Q_{C,\theta, m,n}$. We denote by $S_{C,\theta,m,n}$ the quadric
 $\iota(Q_{C,\theta, m,n})\subset\mP^3$.
 \end{defn}
  
The fact that $m$ and $n$ are general points is important in order the morphism $\Phi:= \phi_{ |\sL_{ \theta+m+n}|} \times \phi_{ |\sL_{W}| } \colon C\to  Q_{C,\theta, m,n}$ to be of degree one. It is also true by the above discussion that the points $n_1,...,n_{d-1}$ such that $\sum_{i=1}^{d-1}n_i\in|\theta +m|$ are all mapped by $\Phi$ to the same point $\iota^{-1}(a_n)$ and similarly $m_1,...,m_{d-1}$ map to the same point $\iota^{-1}(a_m)$. 
We have shown that the image  $\Phi(C)\subset Q_{C,\theta, m,n}$ is a curve of class $|(d,d)|$ with two points of multiplicity $d-1$.

The following Corollary is crucial for our rationality result. Denote by $|(1,0)|$ and $|(0,1)|$ respectively the linear systems given by the two rulings of $Q_{C,m,n}$. From now on we do not distinguish between $\iota^{-1}(a_m), \iota^{-1}(a_n)$ and respectively their $\iota$-images in $S_{C,m,n}$.

\begin{cor}\label{Serveassai} The image $M=\Phi(C)$ is an element of $|(d,d)|$ with two singular points of multiplicity $d-1$, one on $a_m$ and the other on $a_n$. There exist two elements $L,L'\in |(0,1)|$ such that for the partition $W(C)=\{w_1,..w_{g+1}\} \bigsqcup \{w'_{1},..,w'_{g+1}\}$ it holds that $$\Phi(w_1),...,\Phi(w_d)\in L,\,\Phi (w'_1),...,\Phi (w'_d)\in L'.$$ Moreover if $T_{a_m}S_{{C,m,n}}$, $T_{a_n}S_{{C,m,n}}$ are the two tangent hyperplanes to $S_{C,m,n}$ at $a_m$ and respectively $a_n$ it holds that
$T_{a_m}S_{{C,m,n}_\mid Q_{C,m,n}}=\langle a_m, f_{\theta,m,n}(m)\rangle\cup \langle a_m, f_{\theta,m,n}(n')\rangle$ and $T_{a_n}Q_{{C,m,n}_\mid Q_{C,m,n}}=\langle a_n, n\rangle\cup \langle a_n, f_{\theta,m,n}(m')\rangle$, 
where $\langle a_m, f_{\theta,m,n}(m)\rangle,\langle a_n, f_{\theta,m,n}(n)\rangle\in |(1,0)|$ and $\langle a_m, f_{\theta,m,n}(n')\rangle, \langle a_n, f_{\theta,m,n}(m')\rangle$ $\in |(0,1)|$. Finally if $j_{C}\colon  \mP(H^0(C,\sL_{ W})^\vee)\to \mP(H^0(C,\sL_{ W})^\vee)$ is the involution induced by the hyperelliptic involution $J_C\colon C\to C$ then its two fixed points are given by $\pi_W(L)$ and  $\pi_W(L')$ where $\pi_W\colon Q_{C,m,n}\to \mP(H^0(C,\sL_{ W})^\vee$ is the natural projection.
\end{cor}
\begin{proof} The first claims follows by Lemma \ref{quadica} and by Lemma \ref{degreedlinear} $(6)$. Finally an easy computation shows that $2n_1+2n_2+...+2n_{d-1} +n+m'\in |\eta|$ and that $2m_1+2m_2+...+2m_{d-1} +m+n'\in |\eta|$ hence the claim follows trivially by the proof of Lemma \ref{quadica}.\end{proof}

\section{Hyperelliptic curves and the quadric threefold}

We use the geometry of the $3$-fold quadric $Q$ to construct spin hyperelliptic curves with two marked points. First we recall some basic facts on rational curves on $Q$.

\subsection{Lines on the quadric threefold}
We need to recall the basic of the geometry on the $3$-fold quadric.

\subsubsection{The Hilbert scheme of lines of the $3$-fold quadric}

Let $\sK\subset\mP^5$ be the Pl\"ucker embedding of the Grassmannian $\mG (2,4)$ of the $2$-dimensional sub-vector spaces of a $4$-dimensional vector space. It is well-known that $\sK$ is a smooth quadric, called {\it{the Klein quadric}} and that the Hilbert scheme of planes contained inside $\sK$ is a disconnetted union of two components $A^{+}\cup A^{-}$ each of which is isomorphic to a $\mP^3$. Let $Q$ be a general hyperplane section. For any line $l\subset Q$ there exists a unique $[\Pi^+]\in A^{+}$ and a unique $[\Pi^{-}\in A^{-}]$ such that $l\subset \Pi^+$ and $l\subset \Pi^-$. On the other hand since $Q$ is a general one, for any plane of $\sK$, say $[\Pi^+]\in A^{+}$, there exists a unique $l\subset Q$ such that $l\subset \Pi^+$. By the universal property of Hilbert schemes it then follows the following well-known result:
\begin{lem}\label{hilblines} The Hilbert scheme ${\rm{Hilb}}^{Q}_{1}$ of lines of the $3$-fold quadric $Q$ is isomorphich to $\mP^3$. Moreover there exists an isomorphism $\iota\colon {\rm{Hilb}}^{Q}_{1}\to\mP^3$ such that $\Pi\in | \iota^{*}\sO_{\mP^3}(1)|$ iff there exists a line $l\subset Q$ such that
$$
\Pi=\Pi_l:=\{[r]\in {\rm{Hilb}}^{Q}_{1}\mid r=l\, {\rm{or}}\, r\cap l\neq\emptyset\}.
$$
\end{lem}

\noindent 
By Lemma \ref{hilblines} ${\rm{Hilb}}^{Q}_{1}$ is endowed with a null-correlation  $\nabla\colon {\rm{Hilb}}^{Q}_{1}=\mP^3\to \check\mP^3$ given by $[l]\mapsto [\Pi_l]$ see c.f. \cite[Section 3]{CF}.

Since $Q\subset \mP^4$ there is a natural embedding $\zeta\colon {\rm{Hilb}}^{Q}_{1}\to {\rm{Hilb}}^{\mP^4}_{1}=\mG (2,5)$. We will use the following result by Hiroshi Tango:
\begin{thm}\label{tango} We identify $\mG (2,5)$ to its Pl\"ucker embedding inside $\mP^9$. 
The natural embedding $\zeta\colon {\rm{Hilb}}^{Q}_{1}=\mP^3\to \mG (2,5)\subset\mP^9$ is given by the $2$-Veronese embedding of $\mP^3$ inside $\mP^9$. 
Moreover the rank-$2$ vector bundle corresponding to the null-correlation $\nabla$ is the pull-back of the universal rank-$2$ vector bundle on $\mG (2,5)$
\end{thm}
\begin{proof} See \cite[Section 6]{Tan}.
\end{proof}

\subsection{Conics and hyperplane sections}
We maintain notations of the Linear Algebra set up of the Introduction. Here we only recall that if  $q\subset Q$
 is a smooth conic inside the smooth quadric threefold $Q\subset \mP^4=\mP(V)$ and $\mP(W)$
  is the projective plane spanned by $q$ then $\mP({\rm{Ann}}(W))\subset\mP(V^{\vee})$ parameterises the pencil of hyperplanes which contain $\mP(W)$.

\subsection{Construction of the hyperelliptic curve} Let $[H]\in\mP(V^{\vee})$. We assume that:
\begin{enumerate}
 \item $[H]\not\in \check Q$
 \item $[H]\not\in\mP( {\rm{Ann}}(W) )  $
 \end{enumerate}
Set $\{x,y\}:=q\cap H$. We denote by $l_x,l_y\in |(1,0)|$ the two lines of the smooth quadric $Q_{H}$ such that $x\in l_x$ and $y\in l_y$. Analogously we denote by $m_x,n_y\in |(0,1)|$ the two lines of the other ruling such that $x\in m_x$ and $y\in n_y$. Let $R\in |(1,d-1)|$ be a general element. Hence there exist $2d-2$ mutually distinct points $x_1,..., x_{d-1}, y_1,..., y_{d-1}$ such that $R\cap l_x=\{x_1,..., x_{d-1}\}$ and $R\cap l_y=\{y_1,..., y_{d-1}\}$. We consider $C(R)\subset R\times q\simeq \mP^1\times\mP^1$ the $(2,d)$-correspondence given as in the Introduction. We will show that $C(R)$ is hyperelliptic. This will follow by using a dyadic structure intrinsically given by the couple $(Q,q)$. Moreover the couple will give also the partition of set of the Weierstrass points $W(C(R))$ to construct
 $\theta(R)$. 
\subsubsection{The dyadic structure}

We consider the subspace ${\rm{Ann}}(W)\subset V^{\vee}$ given by functionals vanishing over $W$. By construction it holds that the line $\mP({\rm{Ann}}(W))$ intersects the dual quadric ${\check{Q}}$ in two distinct points $[\Pi_z],[\Pi_{z'}]$. This means that there exists two points $z,z'\in Q$ such that $\Pi_z=T_{z}Q$, $\Pi_{z'}=T_{z'}Q$ and $\Pi_z\cap \mP(W)\cap Q=q=\Pi_{z'}\cap \mP(W)\cap Q$. Let $\Delta_H:=\Pi_z\cap Q_H$ and $\Delta_H':=\Pi_{z'}\cap Q_H$. Since $[H]$ is general it holds by direct computation that $\Delta_H$ and $\Delta_H'$ are smooth $(1,1)$ sections of $Q_H$.

\begin{figure}
\centering
\includegraphics[scale=0.8]{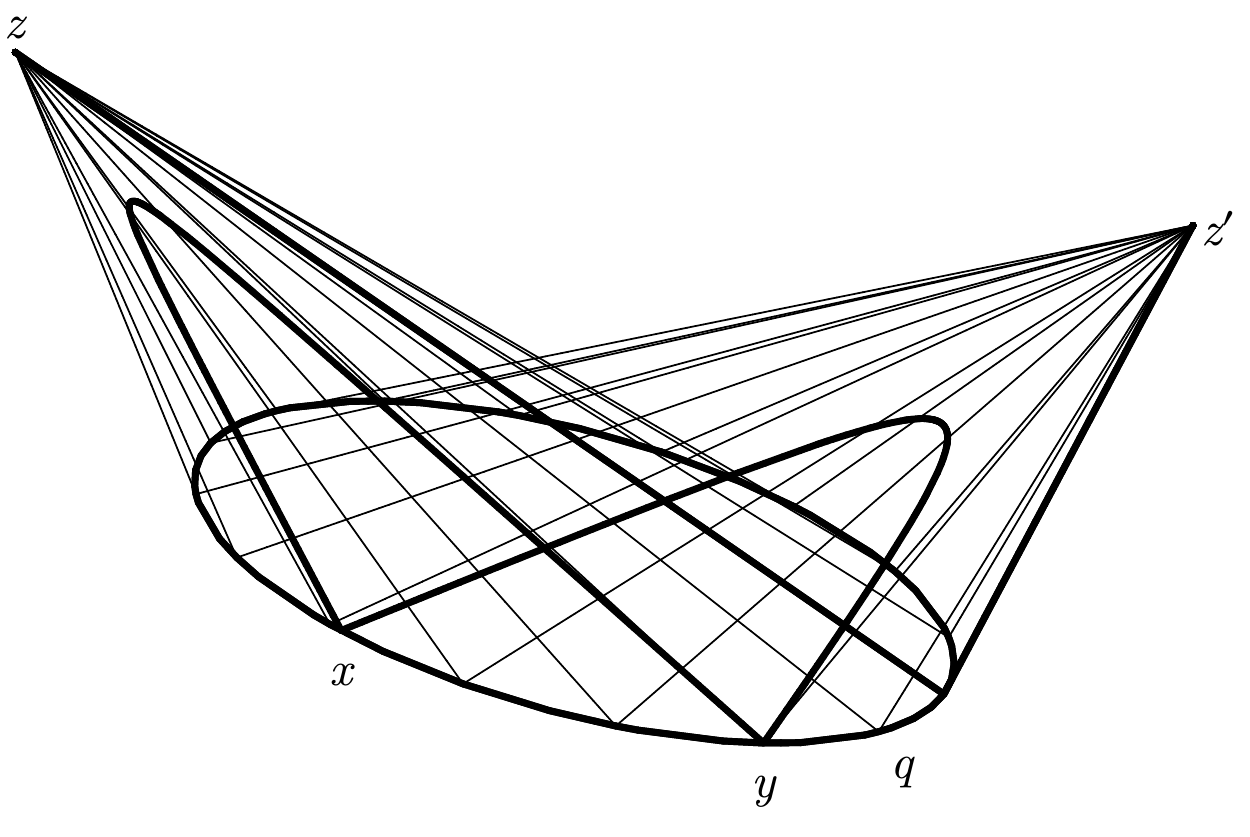}
\caption{The conic $q$}
\label{etichettaPerRiferimenti}
\end{figure}

\begin{defn} We call the curves $\Delta_H, \Delta_H'\subset Q_H$ {\it{the Weierstrass conics of the triple $(Q,q,H)$.}}
\end{defn}

By generality of $R$ it follows also that 
 $\Delta_H\cap R=\{s_1,..., s_d\}$, $\Delta_H'\cap R=\{s'_1,..., s'_d \}$ are transversal intersections. This forces $s_1,..., s_d, s'_1,..., s'_d$ to be $2d$ mutually distinct points.  Moreover let $l_{i}$, $l'_{i}$ be respectively the line 
 $\langle z,s_i\rangle$, $\langle z', s'_i \rangle$. Set $l_i \cap q:=\{\alpha_i \}$ and $l'_i\cap q:=\{\alpha'_i \}$. Then $l_i=\langle z,s_i\rangle=\langle s_i, \alpha_i\rangle$ and $l'_i= \langle z', s'_i \rangle=\langle s'_i,\alpha'_i\rangle$. We have constructed $2d$ points of ${\rm{Hilb}}^Q_1$. Set $w_i:=[s_i,\alpha_i]$, $w'_i:=[s'_i,\alpha'_i]$ for the corresponding points on $C(R)\subset R\times q$, $i=1,..., d$.

\begin{gen}\label{generality conditions}
We define the open subset $ |(1,d-1)|^{oo}\subset |(1,d-1)|^{o}$ given by those $[R]\in  |(1,d-1)|$ which satisfies the following conditions:
\label{gen}

\vspace{5pt}

\begin{enumerate}[(a)]
\item
$\Delta_H\cap l_x \cap R=\emptyset, \Delta_H\cap l_y\cap R=\emptyset$;
\item
$\Delta_H'\cap l_x \cap R=\emptyset, \Delta'_H\cap l_y\cap R=\emptyset$;
\item  $\Delta_H\cap R=\{s_1,..., s_d\}$  is a transversal intersection;
\item$\Delta_H'\cap R=\{s'_1,..., s'_d \}$ is a transversal intersection;
\item
 $R\cap l_x=\{x_1,..., x_{d-1}\}$ is a transversal intersection;
\item
$R\cap l_y=\{y_1,..., y_{d-1}\}$ is a transversal intersection.
\end{enumerate}
\end{gen}
\vspace{5pt}

\begin{lem}\label{hyperellipticity} 
If $R\in |(1,d-1)|^{oo}$ the correspondence $C(R)\subset R\times q$ is a smooth hyperelliptic curve of genus $d-1$.
\end{lem}
\begin{proof} First we stress that $C(R)\in |(2,d)|$ inside $\mP^1\times\mP^1=R\times q$. Moreover there exists a unique line $\langle s_i,\alpha_i\rangle\subset Q$ such that $w_i=[s_i,\alpha_i]\in C$ and the same for $\langle s'_i,\alpha'_i\rangle$, $i=1,..., d$. Hence by construction the points $w_i,w'_i\in C(R)$ are smooth points of $C(R)$ and  they are branch points for the $2$-to-$1$ morphism $\pi_{C(R)}\colon C(R)\to R$. Now let $\nu\colon \widetilde C\to C(R)$ be the normalization morphism and $\pi'\colon  {\widetilde{C}}\to R=\mP^1$ the induced morphism. By Riemann-Roch on $R\times q$ it holds that the arithmetical genus of $C(R)$ is $d-1$. On the other hand by Riemann-Hurwitz's theorem it holds that $g( {\widetilde{C}})\geq d-1$. Hence $\nu\colon \widetilde C\to C(R)$ is an isomorphism and the claim follows.
\end{proof}

\subsubsection{The notion of marked line}
Consider a general point $[t,a]\in C(R)$. By definition the line $\langle t, a\rangle$ is a line of $Q$. We stress that in general $\langle t, a\rangle\not\subset Q_H$.
\begin{defn}
The point $[t,a]$ is called {\it{the marked line}} from $t\in R$ to $a\in q$. We also call the line $l_{[t,a]}=\langle t,a \rangle$ {\it{the support of the marked line}} $[t,a]$. If no confusion arises we sometimes call $l_{[t,a]}\subset Q$ the marked line from $t\in R$ to $a\in q$.
\end{defn}

\begin{rem}\label{multisupporto} There are marked lines with the same support. Indeed denote by $l_x,l_y$ the element of $|(1,0)|$ of $Q_H$ which pass through $x$ and respectively $y$. By generality of $q$, $H$ and $R$ there exist $2d-2$ distinct points $x_1,..., x_{d-1}, y_1,..., y_{d-1}$ such that $R\cap l_x=\{x_1,..., x_{d-1}\}$ and $R\cap l_y=\{y_1,..., y_{d-1}\}$ and for the marked lines $[x_i, x]$, $[y_i, y]$ it holds that $[x_i, x]\neq [x_j, x]$, $[y_i,y]\neq [y_j,y]$, $i,j=1,...,d-1$, $i\neq j$. On the contrary for every $i,j=1,...,d-1$ it holds that $l_{[x_j, x]}=l_{[x_j,x]}=l_x$ and $l_{[y_j, y]}=l_{[y_j,y]}=l_y$. 
\end{rem}
\begin{defn}\label{datonome} $C(R)$ is called {\it{the scheme of the marked lines of $(Q,q,H,R)$}}.
\end{defn}

We can sum up the above results into the following Proposition:
\begin{prop}\label{ssuumm}
Let $Q\subset\mP^4$ be a smooth quadric threefold and let $q\subset Q$ be a smooth conic. If $H\subset\mP^4$ is a general hyperplane with respect to $(Q,q)$ and $R\subset Q\cap H$ is a curve satisfying the generality conditions \ref{generality conditions} then the scheme of the marked lines of $(Q,q,H,R)$ is a smooth hyperelliptic curve $C(R)$ of genus $d-1$.
\end{prop}

\subsubsection{The singular model} By construction we can associate to each marked line a line of $Q$. Let 
$$f_R\colon C(R)\to {\rm{Hilb}}^{Q}_{1}$$
 be the corresponding forgetful morphism, that is the morphism $R\times q\supset  C(R)\ni [t,a]\mapsto [l_{[t,a]}]\in {\rm{Hilb}}^{Q}_{1}=\mP^3$; see: Lemma \ref{hilblines}. We will use the following:
\begin{prop}\label{thesingularmodel} 
Let $M(R):= f_R(C(R))$. Then $M(R)$ is a curve of degree $2d$ with two singular points of multiplicity $d-1$.
\end{prop}
\begin{proof} In this proof we set $f:=f_R$, $C:= C(R)$ and $M:= M(R)$. By Lemma \ref{hilblines} we can and we do identify ${\rm{Hilb}}^{Q}_{1}$ to $\mP^3$. We set $\sL:= f^{*}\sO_{\mP^3}(1)$. Fix a line $l\subset Q$ and set $\Pi:=\Pi_l$. By construction $f^{*}(\Pi)$ is the subscheme of marked lines whose support intersects $l$. We consider the projection from $l$:

\begin{equation}
\label{projectionline}
\xymatrix{ & {\widetilde{Q} \ar[dr]^{  \pi_{l}  }
\ar[dl]_{  {\rho_{l}}}  } \\
{Q} & & \mP^2}
\end{equation}
Since $l$ is general it is easy to see that the $\pi_l$-images $q_l, R_ l\subset\mP^2$ of the $\rho_l$-proper transform of respectively $q,R\subset Q$ are respectively a conic and a rational curve of degree $d$. Let $E_l\subset  {\widetilde{Q}}$ be the exceptional divisor, ${\widetilde{H}}=\rho_{l}^{*}(H)$ and $L\in |\pi_{l}^{*}\sO_{\mP^2}(1)|$. Clearly $L\in |{\widetilde{H}}-E_l|$. Then the fibers of $\pi_l\colon {\widetilde{Q}}\to\mP^2$ are the proper transforms of the lines $r\subset Q$ such that $r\cap l\neq\emptyset$.

 By Bezout's theorem and by generality of $l$ with respect to $R$ and $q$, or even by explicit computation, it holds that $q_l$ intersects $R_l$ in $2d$ distinct smooth points. This implies that there are $2d$ distinct points $$[t_1,a_1],...,[t_{2d},a_{2d}]\in C$$ such that $l\cap l_{[t_i,a_i]}\neq \emptyset$, $i=1,...,2d$. We have shown that $\sum_{i=1}^{2d}[t_i,a_{i}]\leq f^{*}(\Pi)$. We claim that $f^{*}(\Pi)=\sum_{i=1}^{2d}[t_i,a_{i}]$. To show this we choose a particular line of $Q$. Indeed consider the line $l_x$ and set $\Pi_x:=\Pi_{l_x}$. First note that by generality of $R\in |(1, d-1)|$ the points of $R\cap l_x=\{x_1,..., x_{d-1}\}$ are mutually distinct. Moreover let 
$m_x, n_y\in  |(0,1)|$ be the unique element of $ |(0,1)|$ such that $x\in m_x$ and respectively $y\in n_y$. Let $p_x(R)$, $p_y(R)$ be respectively the unique point of $m_x\cap R$ and $n_y\cap R$. Let $\xi_i\in q$ be the point such that $[x_i ,\xi_i]$ and $[x_i,x]$ are the two marked lines which start from $x_i$, $i=1,..., d-1$; we point out that by generality condition \ref{generality conditions} (a) it holds that $\xi_i\neq x$.
 Now note that the line $l_x\subset Q_H$. By the diagram $(\ref{projectionline})$ applied to $l_x$ it follows that $f^{*}(\Pi_x):= [p_x(R), x]+[p_y(R),y]+ \sum_{i=1}^{d-1}[x_i, x]  + \sum_{i=1}^{d-1}  [x_i ,\xi_i]$ since the strict transform of $T_{x_i}Q$ by $\rho_{l_i}\colon {\widetilde {Q}} \to Q$ are smooth surfaces, $i=1,..., d-1$. Hence by simple degree reasons  the claim follows for a general $[l]\in  {\rm{Hilb}}^{Q}_{1}$. 
 
 Now we show that $M$ has two singular points. We recall Remark \ref{multisupporto}. The $d-1$ marked lines $[x_i,x]$ have the same support on the line $l_x$. The same holds for $l_y$. This means that the image $M$ has two multiple points of degree $d-1$ which have support on the point $[l_x]$ and respectively $[l_y]$. By standard normalization theory and by Lemma \ref{hyperellipticity} it follows that $M$ is a curve of degree $2d$ with two singular points of multiplicity $d-1$ and no other singular points. Indeed if we project $\pi_{[l_x]}\colon \mP^3\dashrightarrow \mP^2$ then the composition $\pi_{[l_x]}\circ f\colon C\to \mP^2$ is exactly the morphism given in Lemma \ref{lem:reconst} $(2)$. Hence the claim follows.
\end{proof}

\subsection{Marked lines and ineffective theta characteristics}
We denote by $[t,a']\in C(R)$ the image of $[t,a]$ by the hyperelliptic involution $J_R\colon C(R)\to C(R)$. We also denote by $g^1_{d}(a)=[t,a]+[t_2,a]+...+[t_d, a]$ the divisor obtained by the marked lines which end on $a\in q$.

We want to study the pull-back $f_R^{\star}(\Pi_l)$ where $l=l_{[t,a]}$, that is we consider the marked line $[t,a]\in C(R)$, then we move to ${\rm{Hilb}}^{Q}_{1}$ via the forgetful morphism $f_R\colon C(R)\to {\rm{Hilb}}^{Q}_{1}=\mP^3$, and finally we pull-back the hyperplane section which parameterises the lines of $Q$ which touch $l$. 

By simple check there exists a subdivisor of $f_R^{*}(\Pi_l )$ of the following form: $[t,a']+[t,a]+[t_2,a]+...+[t_d, a]$. By the proof of  Proposition \ref{thesingularmodel} and by the null-correlation  $\nabla\colon {\rm{Hilb}}^{Q}_{1}=\mP^3\to \check\mP^3$ we have that there must exist other 
marked lines, not necessarily distinct, [$z_1,a_1],..., [z_{d-1}, a_{d-1}]\in C(R)$ such that $l\cap l_{[z_i, a_i]}\neq\emptyset$, and such that:
\begin{equation}\label{generaldivisor}
f_R^{*}(\Pi_l )=[z_1,a_1]+...+[z_{d-1}, a_{d-1}]+  [t,a']+[t,a]+[t_2,a]+...+[t_d, a].
\end{equation}
\begin{prop}\label{thetatheta} There exists an ineffective theta characteristic $[\theta(R)]\in {\rm{Pic}}(C(R))$ such that for the general point $[t,a]\in C(R)$, it holds that the unique effective divisor $D_{[t,a]}$ inside $|\theta(R)+[t,a]|$ is the following one: 
$$
D_{[t,a]}= [z_1,a_1]+...+[z_{d-1}, a_{d-1}].
$$
\end{prop}
\begin{proof} We define $D_{[t,a]}:= [z_1,a_1]+...+[z_{d-1}, a_{d-1}]$; the claim is equivalent to show that $\theta\sim D_{[t,a]} -[t,a]$ is an ineffective theta characteristic.  We consider the set of Weierstrass points as obtained in Lemma \ref{hyperellipticity}; 
$w_i:=[s_i,\alpha_i]$, $w'_i:=[s'_i,\alpha'_i]$, $i=1,...,d$. We consider the embedding $\zeta\colon {\rm{Hilb}}^{Q}_{1}=\mP^3\to \mG (2,5)\subset\mP^9$.
 It is easy to show that the images $\zeta([l_i])$ of the supporting  lines $l_i$, of the points $[s_i,\alpha_i]=w_i\in C(R)$, $i=1,...,d$ belong to a conic. Hence by Proposition \ref{tango} all the $[l_i]$ belong to a line $L$ of $\mP^3$. The same holds for the supporting  lines $l'_i$, of $[s'_i,\alpha'_i]$, $i=1,...,d$, and we call $L'$ the corresponding line of $\mP^3$.

 By Lemma \ref{formoftheta} $\theta'\sim \sum_{i=1}^{d}[s_i,\alpha_i]-g^1_2$ is an ineffective theta characteristic. We claim $\theta\sim \theta'$. Indeed we select $l_i=l_{[s_i,\alpha_i]}$ and by the proof of Proposition \ref{thesingularmodel} we have
$f^{*}(\Pi_{l_i})=[s_i,\alpha_i]+[t_1,\alpha_i]+...+[t_{d-1},\alpha_i]+ \sum_{j=1}^{d}[s_j,\alpha_j]$; see also the identity (\ref{generaldivisor}). In other words $f^{*}(\Pi_{l_i})\sim \pi_{2}^{-1}(\alpha_i)+   g^1_2+\theta'$ where $\pi_2\colon C\to q$ is the degree $d$-morphism induced by the projection $R\times q\to q$. 
We turn to equation (\ref{generaldivisor}).
By definition $[t,a]+[t_2,a]+...+[t_d, a]=\pi_2^{-1}(a)$ then
$f^{*}(\Pi_{l})=D_{[t,a]} +[t,a']+\pi_2^{-1}(a)\sim \theta +g^1_2+\pi_2^{-1}(a)\sim f^{*}(\Pi_{l_i})\sim \pi_{2}^{-1}(\alpha_i)+   g^1_2+\theta'$. Hence the claim follows if we put $\theta(R):=\theta$.
\end{proof}
\begin{rem}\label{remarco}
We have seen in Remark \ref{multisupporto}; that is the $d-1$ marked lines $[x_i,x]$ have the same support on $l_x$ and the $d-1$ marked lines $[y_i,y]$ have the same support on $l_y$, $i=1,...,d-1$; see also the proof of Proposition \ref{thesingularmodel}. We point out the reader that the above construction show us four special marked lines:
\begin{enumerate}[{1)}]
\item $n_y(R):=[p_y(R),y]$;
\item $n'_y(R):=[p_y(R),y']$;
\item $m_x(R):=[p_x(R),x]$
\item$m'_x(R):=[p_x(R),x']$
\end{enumerate}
where $\{y,y'\}$ and $\{x,x'\}$ are respectively the intersection of $q$ with $T_{p_{y}(R)}Q$ and $T_{p_{x}(R)}Q$.
\end{rem}

Now by our interpretation of the thetacharacteristic given in  Proposition \ref{thetatheta} we have:
\begin{cor}\label{immaginichiave} Using above notation it holds:
\begin{enumerate}[{1)}]
\item $\sum_{i=1}^{d-1}[x_i,x]\in |\theta(R)+n_y(R)|$;
\item $\sum_{i=1}^{d-1}[y_i,y]\in |\theta(R)+m_x(R)|.$
\end{enumerate}
Moreover by the forgetful morphism $f_R\colon C(R)\to {\rm{Hilb}}^{Q}_{1}$ it holds that:
 $$
 f_R([x_1, x])=f_R([x_2, x])=...=f_R([x_{d-1}, x])=[l_x],
 $$
$$
 f_R([y_1, y])=f_R([y_2, y])=...=f_R([y_{d-1}, y])=[l_y].
 $$
\end{cor}
\begin{proof} The claim follows by Proposition \ref{thetatheta}. 
\end{proof}

To ease reading it is useful to sum up the results of this section.

\begin{figure}
\centering
\includegraphics[scale=0.7]{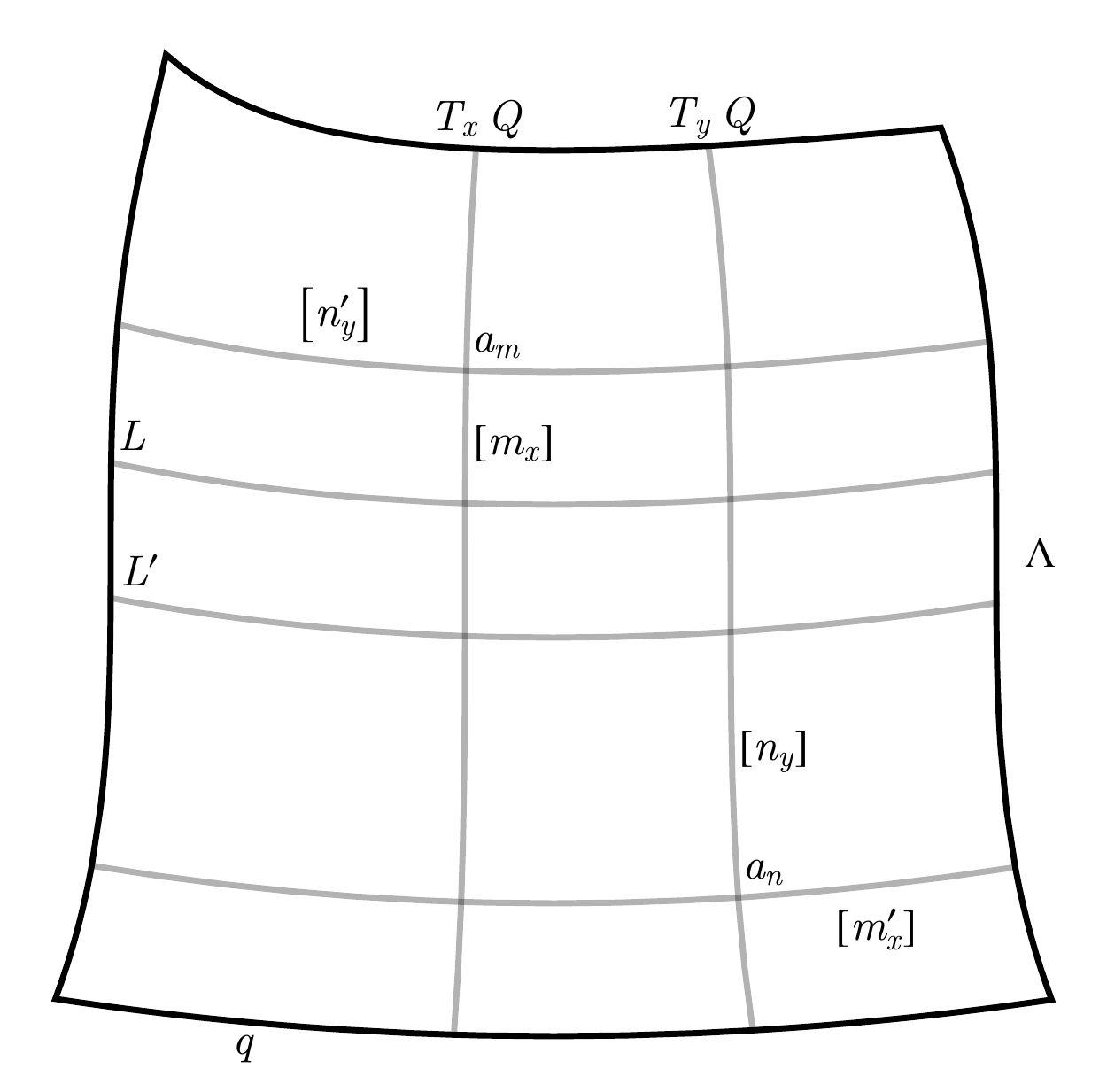}
\caption{The surface $S_q$}
\label{etichettaPerRiferimenti}
\end{figure}

\begin{thm}\label{useusefulful} The scheme of marked lines of the $4$-ple $(Q,q,H, R)$ is an hyperelliptic curve $C(R)$ which comes equipped with an ineffective theta characteristic $\theta(R)$ given by the incidence relation of lines of $Q$. 
The image $M(R)$ of the forgetful morphism $f_R\colon C(R)\to {\rm{Hilb}}^{Q}_{1}$ is contained inside the smooth quadric $S_{q}$ which parameterises the lines of $Q$ which intersect $q$. The two points $[l_x],[l_y]\in {\rm{Hilb}}^{Q}_{1}$ are the support of respectively 
two points of $M(R)$ each of them of multiplicity $d-1$. The rational map $h_R\colon C(R)\to \mP^2$ given by the forgetful morphism followed by the projection $\mP^3\setminus \{[l_x]\}\to \mP^2$ is the morphism $\phi_{|\theta(R)+g^1_2+m_{x}(R)|}\colon C(R)\to\mP^2$ described in Lemma \ref{lem:reconst}.
\end{thm}
\begin{proof} It follows by Proposition \ref{ssuumm}, by Proposition \ref{thesingularmodel} and by Proposition \ref{thetatheta} and its proof. Finally notice that if we take two general line $l,r\subset Q$ then the hyperplane section spanned inside $\mP^4$ by $l$ and $r$ intersects $q$ in two points. Hence the last claim is trivial since the subscheme of ${\rm{Hilb}}^Q_1$ given by lines intersecting $q$ is a smooth $\mP^1$ bundle over $q$ of degree $2$ inside ${\rm{Hilb}}^Q_1=\mP^3$. 
\end{proof}

\subsection{Smooth $2$-dimensional quadrics associated to the scheme of marked lines}
Actually we need to build an identification between the quadric $S_q$ described in Theorem \ref{useusefulful}  and the quadric $S_{C,\theta,m,n}$ of Definition \ref{eccola}, naturally associated to the point $[C,\theta, m,n]\in \sS^{+,{\rm{hyp}}}_{d-1,2}$. We do it first in the case where 
$[C,\theta, m,n]=[C(R),\theta(R), m_x(R),n_y(R)]$.

In Theorem \ref{useusefulful} we have considered the $4$-ple $(Q,q,H, R)$ and  the associated scheme of marked lines $C(R)$. We have set $H\cap q=\{x,y\}$ where $\{x\}=m_x\cap l_x$, $\{y\}=n_y\cap l_y$. We have chosen $R\in |1,d-1|$ and we have set $m_x\cap R=\{p_x(R)\}$ and $n_y\cap R=\{p_y(R)\}$ obtaining two special marked lines: $m(R):=[p_x(R),x]$, $n(R):=[p_y(R),y]$; see also Remark \ref{remarco}.

By Theorem \ref{useusefulful} we can define a rational map $$\nu\colon |(1,d-1)|\dashrightarrow \sS^{+,{\rm{hyp}}}_{d-1,2}$$ given by $$\nu\colon [R]\mapsto \nu([R])=[C(R),\theta(R), m(R),n(R)]\in \sS^{+,{\rm{hyp}}}_{d-1,2}.$$ 
\noindent
By Definition \ref{eccola} 
$$Q_{C(R),m(R),n(R)}:= |\sL_{\theta(R) +m(R)+n(R)}|^{\vee}\times|\sL_W(R)|^{\vee}$$ is the abstract surface naturally associated to the point $\nu([R]) \in \sS^{+,{\rm{hyp}}}_{d-1,2}$ according to its Definition \ref{eccola}. We also denote by  $S_{C(R),m(R),n(R)}\subset\mP^3$ its projective image.
To relate the quadric $S_q\subset  {\rm{Hilb}}^{Q}_{1}$ to the quadric $S_{C(R),m(R),n(R)}$ we first identify $Q_{C(R),m(R),n(R)}$ to an abstract model of $S_q$.
\subsubsection{Abstract definition of $S_q$}
We consider the pencil $\mP({\rm{Ann}}(W))\subset \mP(\check V)$. Let $\rho_{\Lambda}\colon\Lambda\to \mP({\rm{Ann}}(W))$ be the $2$-to-$1$ cover induced by the  $2$-to-$1$ cover of $\mP(\check V)$ branched over the dual quadric $Q^\vee$. This means that a point $([\Pi],\star)\in\Lambda$ is the datum of $[\Pi]\in \mP(\check V)$ and the class of a ruling of the hyperplane section $\Pi\cap Q=Q_{\Pi}$. Let $j_{\Lambda}\colon \Lambda\to \Lambda$ be the involution associated to $\rho_{\Lambda}\colon\Lambda\to \mP({\rm{Ann}}(W))$. Obviously $j_{\Lambda}([\Pi],\star)=([\Pi],-\star)$ where $-\star$ is the class of the other ruling on $Q_{\Lambda}$; it deserves a definition:
\begin{defn}\label{exchanger} The automorphism $j_{\Lambda}\colon \Lambda\to \Lambda$ is called {\it{exchanger of rulings}}.
\end{defn}
Our picture comes with two special points $z,z'\in Q$. Their projective tangent spaces $\Pi_z$, $\Pi_{z'}$ obviously give two points $[\Pi_z] , [\Pi_{z'}]\in \mP({\rm{Ann}}(W))$. By constriction $[\Pi_z]$ and $[\Pi_{z'}]$ are the two branched points of $\rho_{\Lambda}\colon\Lambda\to \mP({\rm{Ann}}(W))$. In the sequel we will consider $[\Pi_z] , [\Pi_{z'}]\in \Lambda$ since for these two points there is no ambiguity regarding the class of rulings to be considered.

\begin{lem}\label{struttura quadricaq} Let $Q_q:=q\times \Lambda$. It holds that the embedding of $Q_q$ as a projective quadrics is $S_q$. Moreover for the two natural fibrations $$\pi_{q}\colon Q_q\to q,$$ $$\pi_{\Lambda}\colon Q_q\to \Lambda$$ it holds:
\begin{enumerate}[{1)}]
\item for every $ a\in q$ the fiber $\pi_{q}^{-1}(a)=\sR_{H_a}$ parameterises the lines of the ruling of the tangent hyperplane section to $Q$ at $a\in q\subset Q$;
\item for every $([\Pi],\star)\in\Lambda$, $\pi_{\Lambda}^{-1} ([\Pi],\star)$ parameterises the lines of the ruling $\star$ (which, obviously, intersect $q$ since $q\subset \Pi\cap Q$).
\end{enumerate}
\end{lem}
\begin{proof} Fix $u\in q$ and let $l\subset Q$ be a line such that $u\in l$. If $T_uQ$ is the tangent hyperplane section to $Q$ at $u$ we see that $l$ gives the point $[l]\in S_q$. Let $\Pi:=\langle q,l\rangle$ be the hyperplane generated by $l$ and $q$. Obviously we have also 
marked  the ruling $\star$ of $Q_{\Pi}:=\Pi\cap Q$ given by $l$. Then $\pi_{\Lambda}([l])=[\Pi,\star]$. The rest is easy by the definition of $S_q$ and of ${\rm{Hilb}}^{Q}_{1}$.
\end{proof}
\subsubsection{Identification of quadrics}
To short we denote $$Q_R:=|\sL_{\theta(R) +m(R)+n(R)}|^{\vee}\times |\sL_W(R)|^{\vee},\, S_R:=S_{C(R),\theta(R),m(R),n(R)}.$$
We are going to identify $Q_R$ to $Q_q$. By Proposition \ref{thesingularmodel} we have a morphism $f_R\colon C(R)\to S_R$ such that 
$M(R):=f_R(C(R))$ is an element of $(d,d)$ with two points $a_{m(R)}$ and $a_{n(R)}$ of multiplicity $d-1$ where $a_{m(R)}$ is the image of the unique effective divisor of $|\theta(R)+n(R)|$ while $a_{n(R)}$ is the image of the unique effective divisor of $|\theta(R)+m(R)|$. We denote by 
$\langle a_{m(R)}, f_R(m(R))\rangle$ and respectively 
$\langle a_{n(R)}, n(R)\rangle$ the corresponding fibers of $\pi_{1,R}\colon Q_R \to |\sL_{\theta(R) +m(R)+n(R)}|^{\vee}$. By construction the two Weierstrass conics $\Delta_H$ and $\Delta_{H'}$ gives respectively the two fibers $L_R$, $L'_R$ of $\pi_{2,R}\colon Q_R\to |\sL_W(R)|^{\vee}$ formed respectively by the two pieces $w_1(R)+...+w_d(R)$ and $w_{1}^{'}(R)+...+w_{d}^{'}(R)$ of the partition $\sP_{\theta(R)}$. We denote by
$$
\Phi_R\colon C(R) \to Q_R
$$ the product morphism, that is $\Phi_R= \phi_{\sL_{\theta(R)+m(R)+n(R)}}\times \phi_{\sL_W(C(R))}$. 
We stress that if $J_R\colon C(R)\to C(R)$ is the hyperelliptic involution then we have two other points $m'(R)=J_R(m(R))$ and $n'(R)=J_R(n(R))$. Finally we think it helps the reader to remind him that we have the isomorphisms 
$\iota_R\colon Q_R\to S_R\subset \mP( H^0(C(R), (\sL_{\theta(R) +m(R)+n(R)})^{\vee}\times H^0(C(R), (\sL_W(R))^{\vee})=\mP^3$, and $\iota_q\colon Q_q \to S_q\subset {\rm{Hilb}} ^{Q}_1$.


\begin{prop}\label{identificazione} Let $C(R)$ be the scheme of marked lines of the $4$-ple $(Q,q,H, R)$. There is a natural identification $$I_{R} \colon Q_{R} \to Q_q(=q\times\Lambda)$$ such that: 
\begin{enumerate}[{1)}]
\item the fiber $\pi_{q}^{-1}(x)=I_R(\pi_{1,R}^{-1}(\pi_{1,R}( \iota_R^{-1}(a_{m(R)})))$
\item the fiber $\pi_{q}^{-1}(y)=  I_R(\pi_{1,R}^{-1}(\pi_{1,R}(\iota_R^{-1}(a_{n(R)})))$
\item the fiber $\pi_{\Lambda}^{-1}[\Pi_z])=I_R(L_R)$
\item the fiber $\pi_{\Lambda}^{-1}[\Pi_{z'}])=I_R (L'_R)$
\item $[l_x] = \iota_q(I_R(\iota_R^{-1}(a_{m(R)})))$
\item $[l_y]=\iota_q(I_R(\iota_R^{-1}(a_{n(R)})))$
\item $[l_{[p_x(R), x']}]= \iota_q(I_R(\iota_R^{-1}(f_R(m'(R)))))$
\item $[l_{[p_y(R), y']}]= \iota_q(I_R(\iota_R^{-1}(f_R(n'(R)))))$
\item $[l_{[p_x(R), x]}]= \iota_q(I_R(\iota_R^{-1}(f_R(m(R)))))$
\item $[l_{[p_y(R), y]}]=  \iota_q(I_R(\iota_R^{-1}(f_R(n(R)))))$
\end{enumerate}
\end{prop}
\begin{proof}
Let $[t,a]\in C(R)\subset R\times q$ be a general marked line. The hyperplane $\Pi:=\langle t,\mP(W)\rangle$ is such that $[\Pi]\in \mP({\rm{Ann}}(W))$. The line $l_{[t,a]} $ identifies only one of the two rulings of $Q_{\Pi}:=Q\cap\Pi$. 
By the same construction we see that via the hyperelliptic involution $j(R)\colon C(R)\to C(R)$, $j(R)\colon ([t,a])\mapsto [t,a']$ the line $l_{[t,a']}$ identifies the other ruling of $Q_{\Pi}$. This makes possible to identify naturally $|\sL_{W(R)}|^{\vee}$ to $\Lambda$. Indeed 
$\sL_{W(R)}=\sO_{C(R)} (\theta(R)+g^1_2)$ hence the fiber containing the support of the unique effective divisor of $|\theta(R)+[t,a] |$ is the one passing through $[t,a']$ and the claim follows by Proposition \ref{thetatheta}. 
By self-explaining notation we have identified $|\sL_{W(R)}|^{\vee}$ to $\Lambda$ in a way such that it holds the following:
\begin{enumerate}[{i)}]
\item $\pi^{-1}_{2,R}(\pi_{2,R}(\Phi_R(m(R))))\leftrightarrow ([\langle\mP(W),p_x(R)\rangle],[\langle p_x(R),x\rangle])$;
\item  $\pi^{-1}_{2,R}(\pi_{2,R}(\Phi_R(n(R)))) \leftrightarrow ([\langle\mP(W),p_y(R)\rangle],[\langle p_y(R),y\rangle])$
\item $\pi^{-1}_{2,R}(\pi_{2,R}(\Phi_R(m'(R)))) \leftrightarrow([\langle\mP(W),p_x(R)\rangle],[\langle p_x(R),x'\rangle])$;
\item $\pi^{-1}_{2,R}(\pi_{2,R}(\Phi_R(n'(R))))\leftrightarrow ([\langle\mP(W),p_y(R)\rangle],[\langle p_y(R),y'\rangle])$;
\item $L_R\leftrightarrow \Pi_z$
\item $L_R' \leftrightarrow \Pi_{z'}$
\end{enumerate}
We consider the embedding $\iota_q\colon Q_q\to S_q$. We want to stress that by our identification it holds that the point $\iota^{-1}_{q}([l_x])\in \pi^{-1}_{2,R}(\pi_{2,R}(\Phi_R(n'(R))$. Indeed for the line $l_x$ it holds that $l_x\subset \langle\mP(W),p_y(R)\rangle$ since its two points $x$ and $l_x\cap l_{[p_y,y]}$ belong to $\Pi=\langle\mP(W),p_y(R)\rangle$. By construction also the line $\langle p_y(R), y'\rangle$ belongs to $Q\cap\Pi$, but since $l_x\cap l_{[p_y,y]}\neq\emptyset$ it holds that $l_x$ belongs to the same ruling of $Q\cap \Pi$ which contains the line $\langle p_y(R), y'\rangle$.  Hence the point $\iota_{q}^{-1}([l_x])\in S_q$ is a point on $\pi^{-1}_{2,R}(\pi_{2,R}(\Phi_R(n'(R))$ and analogously $\iota_{q}^{-1}([l_y])\in \pi^{-1}_{2,R}(\pi_{2,R}(\Phi_R(m'(R))$.

Te quadric $Q_H$ which contains $R$ has two projections $\pi_{i}^{Q_H}\colon Q_H\to\mP^1_{Q_{H}, i }$, $i=1,2$. Since $R\in |(1,d-1)|$ the restriction $\pi_{2_{|R}}^{Q_H}\colon R\to  \mP^1_{Q_{H},2}$ is an isomorphism.
Now we take the embedding $C(R)\subset R\times q$. We consider the composition of the two elementary transformations centred on $[p_x(R),x]\in C(R)$ and respectively on (the strict transform of ) $[p_y(R),y]\in C(R)$  and which maintain the fibration $\pi_q^{R\times q}\colon R\times q\to q$:
\begin{equation}
\label{eq:elementary}
\xymatrix{
& R\times q \ar[dl]_{    \pi_R^{R\times q  }   }\ar[dr]^{  \pi_{q}^{R\times q} }    &\dashrightarrow & 
Q_{R}   \ar[dl]_{    \pi_1^{R }   }    \ar[dl] \ar[dr] ^{  \pi_{2}^{R} }        & \\
 R=\mP^1_{Q_{H},2} &  & q=  |\sL_{\theta(R) +m(R)+n(R)}|^{\vee}& &  |\sL_W(R)|^{\vee} . }
\end{equation}
 
 Now we show that if $a\in q$ and if $H_a$ is the hyperplane section given by the (projective) tangent space to $Q$ at $a$ then $H_a\cap R$ consists on $d$ points, $t_1(a),..., t_d(a)\in R$ such that $\sum_{i=1}^{d}[t_i(a),a]\in |\theta(R)+m(R)+n(R)|$. Indeed this is easy to be checked over the two points $x,y\in H\cap q$ and this shows that the restriction $\pi_{q|C(R)}^{R\times q}\colon C(R)\to q$ is given by the linear system $|\theta(R)+m(R)+n(R)|$. This induces a natural identification $q=  |\sL_{\theta(R) +m(R)+n(R)}|^{\vee}$ and by the previous identification $\Lambda=|\sL_W(R)| ^{\vee}$ the claims $1)$, $2)$, $3)$, $4)$, $5)$, $6)$ follow. It remains to show that also $7)$, $8)$, $9)$, $10)$ hold. 
 We only show $7)$ and $9)$ since the proof of the other claims are analogue. First we show $9)$. Inside $S_q$ the point $[l_{[p_x(R),x]}]$ is given by $\pi_q^{-1}(x)\cap \pi_{\Lambda}^{-1}([\langle\mP(W),p_x\rangle],[\langle p_x(R),x\rangle])$. 
 On the other hand we have that  the unique effective divisor inside $|\theta(R)+m(R)+n(R)|$ which contains $m(R)$ in its support is the one given by the unique effective divisor of $|\theta(R)+n(R)|$ plus $m(R)$, which, by Proposition \ref{thetatheta}, is given by the marked lines through the point $x$. We also have that the unique effective divisor inside $|\sL_W(R)|$ which contains $m(R)$ is given by the hyperplane section $\langle \mP(W),p_x(R)\rangle\cap Q$ where we consider on it the ruling given by $[p_x(R),x]$. Finally we show $7)$. Inside $S_q$ the point $[l_{[p_x,x']}]$ is given by $\pi_q^{-1}(x)\cap \pi_{\Lambda}^{-1}([\langle\mP(W),p_x(R)\rangle],[\langle p_x(R),x'\rangle])$.  On the other hand we have that  the unique effective divisor inside $|\theta(R)+m(R)+n(R)|$ which contains $m'(R)$  is the one given by the marked lines through the point $x'$. We also have that the unique effective divisor inside $|\sL_W(R)|$ which contains $m'(R)$ is given in our interpretation $|\sL_W(R)|^{\vee} \leftrightarrow\Lambda$ by the hyperplane section $\langle \mP(W),p_x(R)\rangle\cap Q$ where we consider on it the ruling given by $[p_x(R),x']$.
 \end{proof}

\section{The Reconstruction Theorem}

\subsection{Reconstruction via space singular models}

In Lemma \ref{degreedlinear} we have constructed a model $M\subset\mP^3$ of $(C,\theta,m,n)$ which shares many common features with the singular model $M(R)\subset{\rm{Hilb}}^{Q}_{1}$ constructed in Proposition \ref{thesingularmodel}. Let $[H]\in\mP(V^{\vee})$ be a general element with respect to $(Q,q)$. We denoted by $\sM_H$ the open subscheme $|(1,d-1)|^{oo}\subset (1,d-1)|$ given by the smooth elements inside $Q_H$ which satisfy the generality conditions \ref{generality conditions}. 

We are ready to show the Reconstruction Theorem:
\begin{thm}\label{reconstruction}
The morphism $\pi_{\sM_H} \colon\sM_H \to \sS^{+,{\rm{hyp}}}_{d-1,2}$ is dominant.
\end{thm}
\begin{proof} The proof is divided in four steps.
\medskip

\noindent
{\it{First Step. Identification of $Q_{C,\theta,m,n}$ to $Q_q$.}} 

Let $[(C,\theta,m,n)]\in  \sS^{+,{\rm{hyp}}}_{d-1,2}$ be a general element. In Corollary \ref{Serveassai} we have constructed the morphism $$\Phi\colon C\to Q_{C,\theta,m,n}=|\sL_{ \theta+m+n}|^\vee\times  |\sL_{ W}|^\vee$$ and its image $M=\Phi(C)$. We recall that the support of $\sum_{i=1}^{d-1}n_i\in |\theta+m|$ is a point $a_m\in M$ of multiplicity $d-1$ and analogously the image $a_n\in M$ of $\sum_{i=1}^{d-1}m_i\in |\theta+n|$ is the other point of multiplicity $d-1$: 
it is not necessary here to distinguish between $Q_C$ and its image inside $\mP^3$ denoted by $S_C$. On the contrary it is better to distinguish between $Q_q=q\times\Lambda$ and its image $\iota_q(Q_q)=S_q\subset{\rm{Hilb}}^{Q}_{1}$. Let us denote 
by $$\pi_{\theta,m,n}\colon Q_C \to |\sL_{ \theta+m+n}|^\vee\,\,{\rm{and}}\,\,\, \pi_W\colon Q_C\to  |\sL_{ W}|^\vee$$
the two natural projections. We know that $\Phi(n)\in \pi_{\theta,m,n}^{-1}(\pi_{\theta,m,n}(a_{n}))$, $\Phi(n)\neq a_n$ and that 
the two points $\pi_W(a_n),\pi_W(\Phi(n))\in |\sL_{ W}|^\vee$ are also distinct. The same holds for $m$ and $a_m$.

We consider the hyperelliptic involution $J_C\colon C\to C$. It induces an involution $j_{C}\colon  |\sL_{ W}|^\vee\to |\sL_{ W}|^\vee$. It holds that $j_C(\pi_W(a_n))=\pi_W(\Phi(m))$ and that  $j_C(\pi_W(a_m))=\pi_W(\Phi(n))$ since Cortollary \ref{Serveassai}. Now we consider also the rulings exchanger
 $j_{\Lambda}\colon \Lambda\to\Lambda$; see Definition \ref{exchanger}. By construction $j_{\Lambda}( \pi_{\Lambda} (   \iota_q^{-1}([l_y])) ) = \pi_{\Lambda} (   \iota_q^{-1}(  [m_x]  ))$ and $j_{\Lambda}(\pi_{\Lambda} (   \iota_q^{-1}([l_x]))) = \pi_{\Lambda} (   \iota_q^{-1}(  [n_y]  ))$. 
 Any two non trivial involution over $\mP^1$ are conjugate.  
 Then $$((|\sL_{ W}|^\vee, j_C),\pi_W(a_n),\pi_W(a_m),\pi_W(\Phi(m)), \pi_W(
 \Phi(n)))$$ can be identified to 
 $$
 ( (\Lambda, j_{\Lambda}),\pi_{\Lambda} (   \iota_q^{-1}([l_y])) , \pi_{\Lambda} (   \iota_q^{-1}([l_x])),   \pi_{\Lambda}(\iota_q^{-1}(  [m_x]  )) ,\pi_{\Lambda}(\iota_q^{-1}(  [n_y] ) )).
 $$
 This forces the identification of {\it{the set}} given by the two fixed points of the rulings exchanger, $[\Pi_z],[\Pi_{z'}]$ to the set given by the fixed points of $j_{C}\colon  |\sL_{ W}|^\vee\to |\sL_{ W}|^\vee$. Now consider the other projection 
 $\pi_{\theta,m,n}\colon Q_C \to |\sL_{ \theta+m+n}|^\vee$ and we can identify $\pi_{\theta,m,n}(a_{m})$ to $x$ and  $\pi_{\theta,m,n}(a_{n})$ to $y$. In other words we identify $(|\sL_{ \theta+m+n}|^\vee, \pi_{\theta,m,n}(a_{m}), \pi_{\theta,m,n}(a_{n}))$ to $(q,x,y)$.
 Since $Q_C= |\sL_{ \theta+m+n}|^\vee\times  |\sL_{ W}|^\vee$ and $Q_q=q\times\Lambda$ we have built an identification $Q_C\leftrightarrow Q_q$ such that:
 \begin{enumerate}
\item $a_m\leftrightarrow \iota_q^{-1}([l_x])$;
\item $a_n \leftrightarrow  \iota_q^{-1}([l_y])$;
\item $\{L,L'\} \leftrightarrow \{\pi_{\Lambda}^{-1}([\Pi_z]),\pi_{\Lambda}^{-1}([\Pi_z'])\}$;
\item $\Phi(m)\leftrightarrow [m_x]$
\item $\Phi(n)\leftrightarrow [n_y]$
\end{enumerate}
\medskip

\noindent
{\it{Second step. Identification of $(\mP^3,S_C)$ to $({\rm{Hilb}}^{Q}_{1}, S_q)$.}}
By the First step we have $Q_{C}\leftrightarrow Q_q$. Then by the $|(1,1)|$ linear system on each one of the two surfaces we can identify $\mP^3$ to ${\rm{Hilb}}^{Q}_{1}$ and respectively $Q_C$ to $S_q$.
\medskip

\noindent
{\it{Third step. Construction of the rational curve $[R]\in \sM$.}} By Lemma \ref{degreedlinear} $(6)$ and by the above identification  of $\mP^3$ to ${\rm{Hilb}}^{Q}_{1}$, the points $f(w_1),...,f(w_d)\in L$ and $f(w'_1),...,f(w'_d)\in L'$ give $d$ lines $l_1,...,l_d$
 of $T_zQ=\Pi_z$ such that $[l_i]=f(w_i)$, and respectively $d$ lines $l'_1,...,l'_d$ of $T_{z'}Q=\Pi_{z'}$ such that $[l'_j]=f(w'_j)$, $i,j=1,...,d$; 
 (we have taken $L\leftrightarrow  \pi_{\Lambda}^{-1}([\Pi_z])$, $L'\leftrightarrow \pi_{\Lambda}^{-1}([\Pi_z'])$, but there is no problem if it were true the opposite case. Actually the curve $C(R)$ has not yet been built; here there is no hidden $\mathbb Z/2\mathbb Z$-action).

 Then there exist $s_1,...,s_d\in \Delta_{H}=Q_{H}\cap T_{z}Q$ and $s'_1,...,s'_d\in \Delta_{H'}=Q_{H}\cap T_{z'}Q$ such that $\{s_i\}=l_i\cap H$ and $\{s'_i\}=l'_i\cap H$.

We claim that there exists a unique $R\subset Q_H$, $R\in |(1, d-1)|$ such that $s_1,...,s_d, s'_1,...,s'_{d-1}\in R$. Indeed consider a curve $R\subset Q_H$ such that $[R]\in |(1, d-1)|$ which passes through $2d-1$, say $s_1,...,s_d, s'_1,...,s'_{d-1}$, among the $2d$ points  $s_1,...,s_d, s'_1,...,s'_d$ of $Q_H$. For every other $R'\in  |(1, d-1)|$ it holds that $R\cdot R'=2d-2$.  Hence $R$ is the unique element of $|(1, d-1)|$ such that $s_1,...,s_d, s'_1,...,s'_{d-1}\in R$.

We claim that $R$ is smooth. Indeed since $[(C,\theta,m,n)]\in  \sS^{+,{\rm{hyp}}}_{d-1,2}$ is general the points $f(w_1),...,f(w_d)$ are general inside $L$ and the points $f(w'_1),...,f(w'_d)$ are general in $L'$. This implies that $l_1,...,l_d$ are $d$ general lines of the cone $Q\cap T_zQ$ and $l'_1,...,l'_d$ are $d$ general lines of $Q'\cap T_{z'}Q$. In particular $s_1,...,s_d$ are $d$ general points of $\Delta_H$ and $s'_1,...,s'_{d-1}$ are general points of $\Delta_{H'}$. This forces $R$ to be smooth otherwise it would be reducible since $\rho_{a}(R)=0$. But if $R$ is reducible this contradicts the fact that the above points of $\Delta_H$ and respectively of $\Delta_{H'}$ are in general position. (Note that we are not claiming that the couple of the two $d$-uples 
$((s_1,...,s_d,), (s'_1,...,s'_{d}))$ is general in $\mathfrak{S}^d(\Delta)\times\mathfrak{S}^{d}(\Delta')$ where $\mathfrak{S}^d(Z)$ is the symmetric product of a variety $Z$). 
\medskip

\noindent
{\it{Fourth step: $C=C(R)$.}} From now on we can identify $Q_q$ to $S_q$. By construction the images $M,M(R)\subset S_q$ pass through the points $[l_x],[l_y], [m_x], [n_y]$. They have a point of multiplicity $d-1$ on $[l_x]$ and on $[l_y]$.

Moreover the line $l_{[p_x, x']}$ and the  image $f(m')$ belong to the fiber of 
$\pi_{\Lambda}\colon S_q\to \Lambda$ which passes through $[l_y]$ and analogously  the line $l_{[p_y, y']}$ and the  image $f(n')$ belong to $\pi_{\Lambda}^{-1}(\pi_{\Lambda}([l_x]))$.
 Now we project $\pi_{ [l_x] } \colon\mP^3\setminus \{ [l_x] \} \to\mP^2$. It is obvious that the morphism $\pi_{[l_x]} \circ f_{C,\theta,m,n}\colon C\to M_{\theta,m}\subset \mP^2$ is exactly the morphism $|\varphi_{\theta+m+g^1_2|}$ of Lemma \ref{lem:reconst} and that 
 $\pi_{[l_x] }\circ f_{C(R),\theta(R),m(R),n(R)}\colon C(R)\to M_{\theta(R),m(R)}\subset \mP^2$  is exactly the morphism given by $|\theta(R)+m(R)+g^1_2(R)|$. In particular by Lemma \ref{lem:reconst} $M_{\theta,m}$ and $M_{ \theta(R),m(R) } $ 
 are two plane curves of degree $d+1$ which have in common the following points $\pi_{[l_x]}([n_y])$ $\pi_{[l_x]}([m_x])$, $\pi_{[l_x]}(f_{C,\theta,m,n}(n'))=\pi_{[l_x]}([l_{[p_y,y']})$. Moreover both have a point of multiplicity $d-1$ on $\pi_{[l_x]}([l_y])$ and they share also $2d-1$ points where they meet tangentially. It holds that 
$$
M_{\theta,m}\cdot M_{\theta(R), m(R)}= (d-1)^2+2(2d-1) +2 +1=(d+1)^2+1.
$$
By Bezout's Theorem $M_{\theta,m}=M_{\theta(R), m(R)}$. Since the restriction of $\pi_{ [l_x] } \colon\mP^3\setminus \{ [l_x] \} \to\mP^2$ to $S_q\setminus \{ [l_x] \}$ is birational then $M=M(R)$. By unicity of the normalisation morphism it holds 
$C=C(R)$ and $f_{ \theta, m, n}=f_{\theta(R), m(R), n(R)}$. Then $\theta=\theta(R)$, $m=m(R)$ and $n=n(R)$.
\end{proof}
\begin{cor} The moduli space of spin hyperelliptic curves with two marked points is irreducible and unirational.
\end{cor}
\begin{proof} Since $\sM_H$ is an open subscheme of a projective space the claim follows by Theorem \ref{reconstruction}.
\end{proof}

\section{The Rationality of $\sS^{+,{\rm{hyp}}}_{g,2}$}
We need to consider a small group action.
\subsection{On the automorphisms of $Q$}
To show our rationality result we need a Lemma on the automorphism group of $Q$.  
\begin{lem}\label{automorfismi}
Let $Q\subset\mP^4$ be a smooth quadric, let $q\subset Q$ be a smooth conic and let $H$ be a general hyperplane with respect to $(Q,q)$. Then 
$${\rm{Aut}}(Q,q,H)\equiv \frac {\mathbb Z}{2\mathbb Z} \times\frac {\mathbb Z}{2\mathbb Z}.
$$
\end{lem}
\begin{proof} We denote by $V:=\mC^5$ the $5$-dimensional vector space such that $Q\subset \mP(V)$. Let $U\subset V$ be the $4$-dimensional sub-vector space such that $\mP(U)=H$ and let $W$ be the $3$-dimensional sub-vector space such that $\mP(W)$ is the space generated by $q$. We fix an equation of $Q$ that is: $Q=V(b_Q)$. We consider $b\colon V\times V\to\mathbb C$ the symmetric bilinear form associated to $b_Q$. Our claim is equivalent to show that 
$$
\widetilde G:=\frac{ \{g\in \mathbb GL(V)\mid \exists \mu\in\mathbb C^{\star} s.t.\,\,g^{\star}b_{Q}=\mu b_{Q}, g(W)=W, g(U)=U\} }{ \{\lambda{\rm{Id}}_{V}\mid \lambda\in\mathbb C^{\star}\}}\
$$ is isomorphich to $\frac {\mathbb Z}{2\mathbb Z} \times\frac {\mathbb Z}{2\mathbb Z}$. By generality assumption the pole $p_{U}$ of $Q$ with respect to $H$ is not in $Q$. This means that the subvector space $U^{\perp_{b}}\subset $V$ $ is generated by 
a vector $u\in V$ such that $b_Q(u)=b(u,u)\neq 0$. In particular it holds that

\begin{equation}\label{U}
V=U\oplus U^{\perp_{b}}.
\end{equation}

Since $q$ is smooth $W\cap W^{\perp_{b}}=\{0\}$. Then
\begin{equation}\label{W}
V=W\oplus W^{\perp_{b}}
\end{equation}
By generality assumptions $u\not\in W\cup W^{\perp_{b}}$, that is:
\begin{equation}\label{uperpwperp}
U^{\perp_{b}}\cap W^{\perp_{b}}=\{0\} 
\end{equation}
and
\begin{equation}\label{uperpw}
U^{\perp_{b}}\cap W=\{0\}.
\end{equation}
Let us fix any $g\in \{g\in \mathbb GL(V)\mid \exists \mu\in\mathbb C^{\star} s.t.\,\,g^{\star}b_{Q}=\mu b_{Q}, g(W)=W, g(U)=U\}$. By our construction there exists $\lambda^{g}_{U^{\perp}}\in\mathbb C^{\star}$ such that 
$$g(u)=\lambda^{g}_{U^{\perp}} \cdot u,$$ and there exist a unique $u_0\in W^\perp$, and a unique $u_1\in W$, $u_i\neq 0$, $i=0,1$ such that
$$
u=u_1+u_0
$$
It holds  that $g(u_1)=\lambda^{g}_{U^{\perp}}\cdot u_1$, $g(u_0)=\lambda^{g}_{U^{\perp}} \cdot u_0$. Note that by generality assumption $u_1,u_0\not\in U$. Then $b_Q(u_i)\neq 0$

By generality assumption the sub-vector space $Z:=U\cap W^\perp$ is of dimension $1$.  Then there exists a non-zero vector $v_1\in Z$ such that 
$Z$ is generated by $v_1$ and $b_Q(v_1)\neq 0$. By construction there exists $\lambda^{g}_1\in \mathbb C^{\star}$ such that $g(v_1)=\lambda^{g}_1\cdot  v_1$. We have shown that
$$W^\perp=\mathbb C\cdot u_0\oplus^{\perp_b} \mathbb C\cdot v_1.$$
Indeed $u\in U^{\perp_{b}}$ and $u_1\in W$ hence
$$0=b(u,v_1)=b(u_1,v_1)+b(u_0,v_1)=0+b(u_0,v_1)=0.
$$

Since $v_1\in U$ and $b_Q(v_1,v_1)=1$, there exists a $3$-dimensional vector space $\widetilde U$ such that 
$$
U=\mathbb C\cdot v_1\oplus^{\perp_b} \widetilde U
$$
and by generality assumption $W\cap\widetilde U$ is generated by a non zero vector $\widetilde u$. By construction  there exists $\lambda^{g}_{\widetilde u}\in \mathbb C^{\star}$ such that $g({\widetilde u})=\lambda^{g}_{\widetilde{u}}\cdot  \widetilde u$. By construction
$W_1:=\mathbb C\cdot u_1\oplus\mathbb C\cdot{\widetilde{u}}$ is a subvector space of $W$ (and in the paper $\mP(W_1)$ is the line between the two points $x,y$ obtained by $H\cap q$). By generality $b(u_1,\widetilde u)\neq 0$. Let $\langle w\rangle \subset W$ be the vector subspace which is orthogonal to $\mathbb C\cdot u_1\oplus\mathbb C\cdot{\widetilde{u}}$. 
By construction there exists $\lambda^{g}_w\in \mathbb C^{\star}$ such that $g(w)=\lambda_w\cdot  w$. Now we put  $v_0:=u_0, v_2:=w, v_3:=u_1, v_4:=\widetilde u$ and we stress that $\sB:=\{ v_0,v_1,v_2,v_3,v_4\}$ is a basis of $V$. The matrix of $g$ with respect to $\sB$ is the following:
$$
[M_{\sB}(g)]= \begin{bmatrix}
    \lambda^{g}_{U^\perp}& 0 &0&0&0\\
    0 &  \lambda^{g}_{v_1}&0&0&0\\
    0&0&\lambda^{g}_w&0&0\\
    0&0&0& \lambda^{g}_{U^\perp}&0\\
     0&0&0&0&\lambda^g_{\widetilde u}\\
 \end{bmatrix}.$$
 Up to rescaling the vectors of $\sB$ we can write
  $$b_Q(v)=  b( u_0,u_0)x_0^2+ x_1^2+x_2^2+b( u_1,u_1)x_3^2+2\alpha x_3\cdot x_4+x_4^2$$ 
  where $v=\sum_{i=0}^{4}x_iv_i$. Since $[g^{\star}b_Q]=[b_Q]$ it holds that that for every such $g$ there exists $\lambda^g\in\mathbb C^{\star}$ such that:
  $$
  (  \lambda^{g}_{v_1}  )^2= ( \lambda^{g}_{w} )^2= ( \lambda^{g}_{\widetilde u}  )^2   = (\lambda^{g}_{ U^{\perp}} ) ^2= \lambda^{g}_{U^{\perp}}\cdot \lambda^{g}_{\widetilde u}=\lambda^g
  $$
  In particular $\lambda^{g}_{ U^{\perp}} = \lambda_{\widetilde u}$. Since in the projective space we can work up to $\pm \lambda$ the claim follows.
\end{proof}

\begin{lem}\label{automorfismirappresento}
The representation $\rho_H\colon \widetilde G\to G\subset {\rm{Aut}}(Q_H)$ is faithful.
\end{lem}
\begin{proof} We use the basis $\sB$ constructed in the proof of Lemma \ref{automorfismi}. We know that the pole of $H$ is the point $[1:0:0:1:0]$. Then $$H:=(b_0x_0+b_1x_3+\alpha x_4=0),$$ where we have set $b_0:=b(v_0,v_0)\neq 0$ and $b_1=b(v_3,v_3)\neq 0$ and by generality assumptions $b_0\neq \mp b_1$.

The group $\widetilde G$ is represented inside $\mP \mathbb G\mathbb L(5,\mathbb C)$ as
$$
\langle {\rm{Id}}, \begin{bmatrix}
1& 0 &0&0&0\\
    0 &  -1&0&0&0\\
    0&0&1&0&0\\
    0&0&0& 1&0\\
     0&0&0&0&1\\
 \end{bmatrix},\begin{bmatrix}
    1& 0 &0&0&0\\
    0 &  1&0&0&0\\
    0&0&-1&0&0\\
    0&0&0& 1\\
     0&0&0&0&1\\
 \end{bmatrix}, \begin{bmatrix}
    1& 0 &0&0&0\\
    0 &  -1&0&0&0\\
    0&0&-1&0&0\\
    0&0&0& 1&0\\
     0&0&0&0&1\\
 \end{bmatrix}\rangle.
$$
We can take coordinates $x_1,x_2,x_3,x_4$  on $H$. The claim now follows by a trivial computation.
\end{proof}
\begin{rem}\label{coordinate} Using notation of the proof of Lemma \ref{automorfismi} and of Lemma \ref{automorfismirappresento} 
we know that $H=(b_0x_0+b_1x_3+\alpha x_4=0)$ and letting $x_0:=-\frac{b_1x_3+x_4}{b_0}$ we can write
$$
Q_H:= (x_1^2+x_2^2+b_1(1+\frac{b_1}{b_0})x_3^2+2\alpha (1+\frac{b_1}{b_0}  )x_3x_4+ (1+\frac{\alpha^2}{b_0})x_4^2=0).
$$
Since $\mP(W)=(x_0=x_1=0)$ we also have that  $$q:=(x_2^2+b_1x_3^2+2\alpha x_3\cdot x_4+x_4^2=0)$$
and $\mP({\rm{Ann}}(W))$ is the pencil $\mu_0\cdot x_0+\mu_1\cdot x_1$ where $[\mu_0:\mu_1]\in\mP^1_{\mathbb C}$. We also have that $\Pi_z=(x_1+i\sqrt{b_0}x_0=0)$, $\Pi_{z'}=(x_1-i\sqrt{b_0}x_0=0)$ and that:
$$
z:= \begin{bmatrix}
    1\\
    -i\sqrt{b_0} \\
    0\\
    0\\
    0\\
 \end{bmatrix},z':= \begin{bmatrix}
    1\\
    i\sqrt{b_0}  \\
    0\\
    0\\
    0\\
 \end{bmatrix}, x:= \begin{bmatrix}
    0\\
    0\\
    \beta \\
    1\\
    -\frac{b_1}{\alpha}\\
 \end{bmatrix}, y:= \begin{bmatrix}
    0\\
    0\\
    -\beta\\
    1 \\
    -\frac{b_1}{\alpha}\\
 \end{bmatrix},
$$
where $\beta=\sqrt{b_1(\frac{b_1}{\alpha}-1)}$
\end{rem}

\begin{rem}\label{Gastratto}
By a trivial coordinate change on $H$ we can assume that $Q_H=y_0y_1-y_2y_3$ and that if $\phi_{H}\colon\mP^1_{[t_0:t_1]}\times \mP^1_{[s_0:s_1]}\to Q_H$ is the Segre embedding then $x$ is the image of $([1:0],[1:0] )$, $y$ is the image of $ ([0:1],[0:1])$ and the group $G$ inside ${\rm{Aut}}(\mP^1_{[t_0:t_1]}\times \mP^1_{[s_0:s_1]})$
is generated by $g,h\in {\rm{Aut}}(\mP^1_{[t_0:t_1]}\times \mP^1_{[s_0:s_1]})$ where
\begin{center}
\begin{tabular}{lll}
$([t_0:t_1],[s_0:s_1])$ $\stackrel{g}{\rightarrow}$ & $([t_0:-t_1],[s_0:-s_1])$ \\
$([t_0:t_1],[s_0:s_1])$ $\stackrel{h}{\rightarrow}$ &$ ([s_1:s_0],[t_1:t_0])$
\end{tabular}
\end{center}
\end{rem}
\begin{proof} It is easy by Remark \ref{coordinate}.
\end{proof}

\subsection{The action on $|(1,d-1)|$}

By Remark \ref{Gastratto} we see that the subgroup of $G$ which does not exchange the rulings is 
$$
G':=\langle g\rangle.
$$
We recall that we have denoted by $\sM_H$ the open subscheme $|(1,d-1)|^{oo}\subset (1,d-1)|$ given by the smooth elements inside $Q_H$ which satisfy the generality conditions \ref{generality conditions}. 
We need to compute the action of the group $G'$ on $\sM_H$.
\begin{lem}\label{rarrar} Let $\mP^{2r+1}:=\mP(H^0(Q_H,\sO_{Q_H}(l_1+rl_2))^{\vee})$ where $l_i$ is the fiber of the natural projection $\pi^{H}_{i}\colon Q_H\to\mP^1$, $i=1,2$. Then $\mP^{2r+1}//G'$ is a rational variety.
\end{lem}
\begin{proof} For any $\sigma\in H^0(Q_H,\sO_{Q_H}(l_1+rl_2))$ there exists unique $a_0,...,a_r; b_0,...b_r\in\mathbb C$ such that 
$$\sigma=t_0\cdot \sum_{i=0}^{r}a _is_0^is_1^{r-i}+t_1\cdot \sum_{i=0}^{r}b_is_0^is_1^{r-i}
$$
The induced action of $g$ on $[a_0:a_1:,...,:a_r:b_0:b_1:...,:b_r]\in\mP^{2r+1}$ is given as follows:
$$
\begin{bmatrix}
    a_0& a_1  & ...  &a_r   &b_0   &b_1  &...& b_r \\
    \downarrow& \downarrow & ...  &\downarrow  &\downarrow &\downarrow &...& \downarrow\\
   (-1)^r  a_0& (-1)^{r-1}a_1  & ...  &a_r   &(-1)^{r+1}b_0   &(-1)^{r}b_1  &...& -b_r \\
 \end{bmatrix}
 $$
since Remark \ref{Gastratto}.
Then the quotient variety is a cone over a Veronese embedding of $\mP^r$. In particular it is a rational variety.
\end{proof}

\begin{cor}\label{loe} $\sM_H//G'$ is a rational variety of dimension $2d-1$.
\end{cor}
\begin{proof} Trivial by Lemma \ref{rarrar}.

 \end{proof}

%
\subsection{The injectivity result}
\noindent In Theorem \ref{reconstruction} we have shown that he morphism $\pi_{\sM_H}\colon\sM_H \to \sS^{+,{\rm{hyp}}}_{d-1,2}$ is dominant. By construction we see that it induces a morphism $$\pi\colon \sM_H//G'\to \sS^{+,{\rm{hyp}}}_{d-1,2}.$$ Now to finish
 we need to show that the above morphism is injective. We need to make explicit the generality conditions on $H$. We denote by  ${\rm{Pol}}_Q([H])$ the pole of $H$ with respect to $Q$. Using the notation of the proof of Lemma \ref{automorfismi} we can write: ${\rm{Pol}}_Q([H])=\langle u\rangle$.
\begin{gen}\label{generalityconditionsonH} We use the notation of Lemma \ref{automorfismi}. From now on $[H]\in\mP(V^{\vee})$ satisfies the following generality contitions:
\begin{enumerate}
 \item $[H]\not\in \check Q$;
 \item $[H]\not\in\mP({\rm{Ann}}(W))$;
 \item ${\rm{Pol}}_Q([H])\not\in \mP(W)$;
 \item ${\rm{Pol}}_Q([H])\not\in \mP(W^{\perp})$;
 \item $\{\langle v_1\rangle\}=\mP(U)\cap \mP(W^{\perp})$;
 \item $\mP(\widetilde U)\cap\mP(W)=\{\langle v_2\rangle\}$ where $U=\langle v_1\rangle\oplus^{\perp_b}\widetilde U$;
 \item The points $z, z', x, y, \langle v_2\rangle\in\mP(V)$ are in general position.
 \end{enumerate}
 \end{gen}
 
Now we can show the Injectivity Theorem:
\begin{thm}\label{injectivity}
Let $[H]\in\mP(V^{\vee})$ be a point satisfying the generality conditions of Remark \ref{generalityconditionsonH}. Let $\sM$ be the open subscheme $|(1,d-1)|^{oo}\subset (1,d-1)|$ given by the smooth elements inside $Q_H$ which satisfy the generality conditions \ref{generality conditions}.
Then the morphism $\pi\colon  \sM//G'  \to\sS^{+,{\rm{hyp}}}_{d-1,2}$ is generically injective. \end{thm}
\begin{proof}
Let  $[R_1]\in\sM$ be a general element. Assume that there exists an $R_2$ such that $\pi_{\sM}([R_1])=\pi_{\sM}([R_2])$. This means that  there exists an automorphism $\chi\colon C(R_1)\to C(R_2)$ such that $\chi^{\star}(\theta(R_2))=\theta(R_1)$, 
$\chi^{\star}(m(R_2))=m(R_1)$, $\chi^{\star}(n(R_2))=n(R_1)$. In particular for the partitions giving the thetacharacteristic it holds that $\chi^{\star}\sP_{\theta(R_2)}=\sP_{\theta(R_1)}$. 
By Theorem \ref{reconstruction}  $C(R_1)$ is a general hyperelliptic curve and the general hyperelliptic curve has only one non trivial automorphism. Snce $m(R_1)+n(R_1)$ is not the hyperelliptic linear system $\chi$ is unique.

We recall that 
$Q(R_i)= |\sL_{\theta(R_i) +m(R_1)+n(R_i)}|^{\vee}\times |\sL_W(R_i)|^{\vee})$, $i=1,2$ and that $Q_q=q\times\Lambda$. Then there is an isomorphism given by the obvious isomorphism on each factor; we still denote it by $\chi\colon Q_{R_1}\to Q_{R_2}$. 
By the construction of $C(R_1)$ and of $C(R_2)$ as schemes of marked lines and by Proposition \ref{identificazione} it induces an automorphism $\chi_{q\times\Lambda}\colon Q_q\to Q_q$ where $\chi_{q\times\Lambda}=I_{R_2}\circ \chi \circ(I_{R_1})^{-1}$. We denote by
 $\chi \colon S_q\to S_q$ the induced automorphism and we have that the following diagram is commutative:

\begin{equation}
\xymatrix {  &C(R_2)\ar[r]^{\Phi_{R_2}} & Q_{R_2}\ar[r]^{ I_{R_2}} & Q_q \ar[r]&S_q\\
 &C(R_1)  \ar[r]^{\Phi_{R_1}} \ar[u]^{\chi} &Q_{R_1} \ar[r]^{ I_{R_1}}    \ar[u]^{\chi}&Q_q  \ar[u]^{\chi_{q\times\Lambda}} \ar[r]&S_q \ar[u]^{\chi},
}
\end{equation}
We definitely have: $\chi ([l_x])=[l_x]$,
$\chi ([l_y])=[l_y]$. We can set $L:=\pi_{\Lambda}^{-1}(\Pi_z)$, $L':=\pi_{\Lambda}^{-1}(\Pi_z')$ and it holds that $\chi_{q\times\Lambda}(L)=L$ and $\chi_{q\times\Lambda}(L')=L'$ since Proposition \ref{identificazione}. We claim that 
$\chi_{q\times\Lambda}=\chi_q\times{\rm{Id}}_{\Lambda}$. Indeed by Proposition \ref{identificazione}  the four points $[l_x]$, $[l_y]$, $[m_x]$ and $[n_y]$ are fix points of $\chi_{q\times\Lambda}$. Moreover their $\pi_\Lambda$-images are four distinct points of $\Lambda$ and they are fixed by the induced automorphism. This implies the claim. 
Now we study the automorphism $\chi_q\colon q\to q$. By its construction we know that if $t\in R_1$ and $[t,\alpha],[t,\alpha']\in C(R_1)$ then there esists $s\in R_2$ such that $[s, \chi_q(\alpha)],[s,\chi(\alpha')]\in C(R_2)$. 
Moreover thank to the two identifications $I_{R_{1}}$ and $I_{R_{2}}$ we actually know that if $\Pi$ is the hyperplane section $\langle \mP(W), t\rangle$ then $l_{[t,\alpha]}$, $l_{[t,\alpha']}$ are two lines inside $Q_\Pi=Q\cap \Pi$. 
Obviously they belong to distinct rulings of $Q_\Pi$ since they meet on $t$. The same holds for $l_{[s, \chi_q(\alpha)]}$, $l_{[s,\chi(\alpha')]}$. Set theoretically we have defined a map $R_1\ni t\stackrel{\phi}{\mapsto} s \in R_2$. The same trick enable us to define a map $\phi\colon Q\to Q$. 
Indeed let
$p\in Q$ and let $\alpha(p),\alpha'(p)\in q$ be the two points of intersection between $q$ and $T_pQ$. We define as above $\Pi:=\langle \mP(W), p\rangle$. Obviously the two lines $\langle p,\alpha(p)\rangle$, $\langle p,\alpha'(p)\rangle$ belongs to distinct rulings of $Q_\Pi$. Now consider the two points $\chi_q(\alpha(p)), \chi_q(\alpha'(p))\in q$ and consider the two lines $l,l'\subset Q_\Pi$ such that $\chi_q(\alpha(p))\in l$, $\chi_q(\alpha'(p))\in l'$ and such that $l$ belongs to the same ruling of $\langle p,\alpha(p)\rangle$, while $l'$ belongs to the same ruling of $\langle p,\alpha'(p)\rangle$. We define $\phi(p)$ to be the unique intersection point of $l$ and $l'$. By definition $\phi\colon Q\to Q$ and it extends $\phi\colon R_1\to R_2$. By construction $\phi\colon Q\to Q$ sends lines to lines. Then $\phi\in {\rm{Aut}}(Q,q,H)$.
By the proof of Lemma \ref{automorfismi} and by its construction we see that it fixes the points
$z,z'$ and $\phi_{}(l_x)=l_x$, $\phi_{}(l_y)=l_y$. Then $\phi_{|Q_H}\in G'$. This shows the claim.
\end{proof}

\subsection{The rationality result}
Finally we can put together our previous results to show:
\begin{thm}\label{theend}
 $\sS^{+,{\rm{hyp}}}_{g,2}$ is a rational variety.
\end{thm}
\begin{proof}
It follows straightly by the Reconstruction Theorem \ref{reconstruction}, by Theorem \ref{injectivity} and by Corollary \ref{loe}.
\end{proof}

\end{document}